\documentclass[]{amsart}
\usepackage{color,latexsym,amssymb,amsmath,amsthm}
\usepackage{graphicx}
\newtheorem{thm}{Theorem}
\newtheorem{lemma}{Lemma}
\newtheorem{propo}{Proposition}
\newtheorem{coro}{Corollary}
\theoremstyle{definition}
\newtheorem*{defn}{Definition}
\def\ker{{\rm ker}} \def\ran{{\rm ran}}   
\def\cat{{\rm cat}} 
\def\tcat{{\rm tcat}} 
\def\gcat{{\rm gcat}} \def\Cat{{\rm Cat}}
\def\ccup{{\rm cup}}
\def\cri{{\rm cri}} 
\def\crit{{\rm crit}} 
\def\gcrit{{\rm gcrit}} 

\def\G{{\mathcal G}}  
\def\R{\mathbf R}

\title{The Lusternik-Schnirelmann theorem for graphs}
\author{Frank Josellis and Oliver Knill}
\address{frank.josellis@senax.net, knill@math.harvard.edu }

\date{November 13, 2012}


\begin{document}
\maketitle

\begin{abstract}
We prove the discrete Lusternik-Schnirelmann theorem $\tcat(G) \leq \crit(G)$ 
for general simple graphs $G=(V,E)$. It relates
$\tcat(G)$, the minimal number of in $G$ contractible graphs covering $G$, with 
$\crit(G)$, the minimal number of critical points which an injective function $f:V \to R$ 
can have. Also the cup length estimate $\ccup(G) \leq \tcat(G)$ is
valid for any finite simple graph. Let $\cat(G)$ be the minimal $\tcat(H)$ among all $H$
homotopic to $G$ and let $\cri(G)$ be the minimal $\crit(H)$ among all graphs $H$ homotopic to $G$,
then $\ccup(G) \leq \cat(G) \leq \cri(G)$ relates three homotopy invariants for graphs:
the algebraic $\ccup(G)$, the topological $\cat(G)$ and the analytic $\cri(G)$. 
\end{abstract}

\section{Introduction}

Developed at the same time than Morse theory \cite{Morse29,Mil63}, 
Lusternik-Schnirelmann theory \cite{LS,CLOT} complements Morse theory. It
is used in rather general topological setups and works also in infinite dimensional situations 
which are used in the calculus of variations. While Morse theory is stronger when applicable, 
Lusternik-Schnirelmann theory is more flexible and applies in more general situations.
We see here that it comes naturally in graph theory. 
In order to adapt the Lusternik-Schnirelmann theorem to graph theory,
notions of ``contractibility'', ``homotopy'',``cup length'' and ``critical points''
must be carried over from the continuum to the discrete. The first two concepts have been
defined by Ivashchenko \cite{I93,I94}. It is related to notions put forward earlier
by Alexander \cite{Alexander} and Whitehead \cite{Whitehead} (see \cite{Cohen}). 
who defined ``simple homotopy'' using ``elementary contractions'' and ``elementary expansions'' 
or ``subdivisions''.
We use a simplified but equivalent definition used in \cite{CYY}, who noted that edge removal and addition 
can be realized with pyramid vertex additions and removals. This is an essential simplification because
this allows to describe homotopy contractions
using injective functions $f$ as $\{ f \leq c_k \; \} \to \{ f \leq c_{k-1} \; \}$. 
Allowing contractions and expansions together produces homotopy.
Functions enter the picture in the same way as in the continuum. 
A smooth function $f$ on a manifold $M$ defines
a filtration $M(c) = \{ f \leq c \; \}$ which for general $c$ is a manifold with boundary. This one parameter family 
of manifolds crosses homotopy thresholds at critical points of $c$. 
This idea can be pushed over to the discrete:
any injective function $f$ on the vertex set $V$ of a graph $G=(V,E)$ orders the vertex set
$f(x_1)<f(x_2)< \cdots < f(x_n)$ and
defines so a sequence of graphs $G_f(x_j)$ generated by 
the vertices $\{ v \in V \; | \; f(v) < f(x) \; \}$. 
Analogously to the continuum, this graph filtration
$G_f(x_j)$ gives structure to the graph and builds up any simple graph starting 
from a single point. While $G(x_j)$ and $G(x_{j+1})$ are often 
Ivashchenko homotopic, there are steps, 
where the homotopy changes. Quite drastic topological changes can happen at vertices $x$, 
when the Euler characteristic of $G(x)$ changes. This 
corresponds to the addition of a critical point with nonzero index. One can
however also encounter critical points of zero index which do not necessarily 
lead to an Euler characteristic change. Category can track this.
As in the continuum, homotopy for graphs is an
equivalence relation for graphs. It can be seen either directly or as a consequence of 
the Euler-Poincar\'e formula, that Euler characteristic is a homotopy invariant.
One has also in the discrete to distinguish
``contractibility'' with ``homotopic to a point'' which can be different. 
It can be necessary to expand a space first before being able to collapse it to a point.
A first example illustrating this in the discrete was given in \cite{CYY}.
We take a definition of contractibility, which is equivalent to the one given 
by Ivashchenko for graphs and which is close to what one can look at 
in the continuum or simplicial complexes. 
\footnote{While pioneers like Whitehead would have considered a graph as a 
one-dimensional simplicial complex which is only contractible, 
if it is a tree, Ivashchenko's definition is made for graphs with no 
reference to simplicial complexes or topological spaces.} 
A topological space $M$ can be called {\bf contractible in itself}, 
if there is a continuous function $f$ on $M$ such that 
$M(x) = \{ y \; | \; f(y) < f(x) \; \}$ are all homotopic or empty. 
This implies that $M$ is homotopic to a point and shares with the later
quantitative topological properties like category or Euler characteristic. Of course, this notion does not 
directly apply in the discrete because for discrete topological spaces, the addition of a 
discrete point changes the homotopy notion when taken verbatim from the continuum. 
The homotopy definition can be adapted however in a meaningful way to graphs, as
Ivashchenko \cite{I93,I94} has shown. As in the continuum, it 
will be important however to distinguish between ``contractibility in itself'' and 
``contractibility within a larger graph $G$''. \\

Also the notion of critical points for a continuous function can be deduced from
the continuum. For a function which lacks differentiability like a function on a
metric space, we can still define critical points. To do so, define a point to be a {\bf regular point} if 
or sufficiently small $r>0$, the sets $S_r^-(x) = \{ d(y,x) = r, f(y) < f(x) \; \}$ are contractible 
in themselves.  All classically regular points of a differentiable function on a manifold 
are regular points in this more general sense.  
The notion of contractibility allows therefore to define critical points also in nonsmooth situations. 
The definition of contractibility for graphs is inductive and as in the continuum, we either can look at
contractibility within itself or contractibility of a subgraph within a larger graph. It will 
lead to the equivalence relation of homotopy for graphs.
A simple closed curve on a simply connected space $M$ for example is contractible in $M$ 
but not contractible in itself nor homotopic to a point. Indeed, 
the three notions ``contractible in itself'', ``contractible within a larger graph $G$'' and
``homotopic to a point'' are all different from each other. Seeing the difference is essential 
also in the continuum. All three notions are important: ``contractible in itself'' is used to
define critical points and simple homotopy steps'', ``contractible within a larger graph $G$''
is used in the definition of category. ``Homotopy'' finally is the frame work which produces natural 
equivalence classes of graphs. \\

The inductive definition goes as follows: a simple graph $G=(V,E)$ is 
{\bf contractible in itself} if there is an injective function $f$ on $V$ 
such that all sub graphs 
$S^-(x)$ generated by $\{ y \in S(x) \; | \; f(y) < f(x) \; \}$ are contractible. 
Only at the global minimum $x$, the set $S^-(x)$ is empty. 
The {\bf geometric category} $\gcat(G)$ of a graph $G$ is the smallest number of in itself 
contractible graphs $G_j=(V_j,E_j)$ 
which cover $G=(V,E)$ in the sense $\bigcup_j V_j = V$ and $\bigcup E_j = E$. 
As in the continuum, the geometric category is not a homotopy invariant. \\
Contractible sets have Euler characteristic $1$ and can be built up 
from a single point by a sequence $G(x_1)=\emptyset,G(x_1)=\{x_1\}, \dots ,G(x_n)=G$ 
of graphs which all have the same homotopy properties since $G(x_{k+1})$ 
is obtained from $G(x_k)$ by a pyramid construction on a contractible subgraph. 
Examples of properties which are preserved are Euler characteristic, cohomology, cup length or category. 
These are homotopy invariants as in the continuum. Two graphs $G=G_0,H=G_n$ are homotopic if one can find
a sequence of other graphs $G_i$ such that $G_{i+1}$ contracts to $G_i$ or $G_i$ contracts to $G_{i+1}$. 
As in the continuum,  topological properties like dimension or geometric category are not preserved 
under homotopy transformations. A graph which is in itself contractible 
is homotopic to a single point but there are graphs homotopic to 
a single point which need first to be homotopically enlarged before one can reduce them to a
point. These graphs are not contractible. We should mention that terminology 
in the continuum often uses ``contractibility'' is a synonym for ``homotopic to a point''
and ``collapsibility'' for homotopies which only use reductions. We do not think there is a danger
of confusions here.  A subgraph $H$ of $G$ is 
{\bf contractible in $G$} if it is contractible in itself 
or if there is a contraction of $G$ such that $H$ becomes a single point. 
Any subgraph $H$ of a contractible graph $G$ is contractible. In the continuum, a closed
curve in a three sphere is contractible because the sphere is simply
connected, but the closed loop is not contractible in itself. This is the same in the discrete: 
a closed loop $C_5$ is a subgraph of $K_5$ which is
contractible within $K_5$ because $K_5$ can be collapsed to a point, but $C_5$ in itself is
not contractible. 
The smallest number of {\bf in $G$ contractible} subgraphs of $G$ 
whose union covers $G$ is called the {\bf topological category} of $G$ and denoted $\tcat(G)$. The minimal
$\tcat(H)$ among all $H$ homotopic to $G$ is called {\bf category} and is a homotopy invariant.
We will show that the topological Lusternik-Schnirelmann category $\tcat(G)$ is
bounded above by the minimal number $\crit(G)$ of critical points by verifying
that for the gradient sequence $G_f(x_j)$, topological changes occur at critical points $x_j$.
Indeed, we will see that if $x$ is a regular point then the category does not change and
consequently that if the category changes (obviously maximally by one), then we have a critical point. 
From $\tcat(G) \leq \crit(G)$ and $\cat(G) \leq \tcat(G)$ we get $\cat(G) \leq \crit(G)$ and 
since the left hand side is a category invariant, we have $\cat(G) \leq \cri(G)$. \\

As mentioned earlier, category theory has Morse theory as a brother. Morse theory makes assumptions
on functions. One of them implies that critical points have index $-1$ or $1$. The index of critical
points has been defined in \cite{poincarehopf}.
An Euler characteristic change implies that the critical point has nonzero index 
$i_f(x) = 1-\chi(S^-_f(x))$, where $S^-_f(x)$ is the graph generated by
$\{ y \in V \; | \; (x,y) \in E, f(y)<f(x) \; \}$. 
The Lusternik-Schnirelmann theorem $\cat(G) \leq \crit(G)$ 
bounds the category above by the minimal number $\crit(G)$ of 
critical points among all injective functions $f$. It is in the following way related to the 
Poincar\'e-Hopf theorem for the graph $G_f(a) = (V(a),E(a))$
$$   \chi(G_f(a)) = \sum_{x \in V(a)} i_f(x) \; . $$
In both cases,  Lusternik-Schnirelmann as well as Poincar\'e-Hopf, the left hand side is a homotopy
invariant while the right hand side uses notions which are not but which combine to a
homotopy invariant. 
The Lusterik-Schnirelmann theorem throws a loser but wider net than Euler-Poincar\'e
because unlike Euler characteristic, category can see also ``degenerate'' critical points of zero index. 
The analogy can be pushed a bit more: the category difference $k_f(x_j)=\tcat(G(x_j))-\tcat(G(x_{j-1}))$ 
defines a {\bf category index} of a point and almost by definition, 
the Poincar\'e-Hopf type category formula 
$$   \tcat(G) = \sum_{x \in V} k_f(x) \;  $$ 
holds. Since the expectation value of the index $K(x)={\rm E}[i_f(x)]$ is curvature \cite{indexexpectation}
we can define a category curvature $\kappa(x) = {\rm E}[k_f(x)]$ which is now independent of functions
and have a Gauss-Bonnet type theorem 
$$   \tcat(G) = \sum_{x \in V} \kappa(x) \;  $$
Similarly, $\cat(G) = \sum_{x \in V} C(x)$ for some function $C(x)$ on the vertices. \\

The analogue, that $\chi(G)$ is bounded above by the number of critical points with nonzero index
is the Lusternik-Schnirelmann theorem $\cat(G) \leq \crit(G)$.  
Viewed in this way, the Lusternik-Schnirelmann theorem is also a sibling of a weak Morse inequality: 
if $f$ is called a {\bf nondegenerate function} on a graph if $i_f(x)=1$ or 
$i_f(x)=-1$ for all critical points $x \in V$ and if $\crit(G)$ is the minimal number 
of critical points a Morse function on $G$ can have, 
then $\chi(G) \leq \crit(G)$.  \\

In order to prove strong Morse inequalities, we need to define a Morse index at critical points. 
This is possible as mentioned in section 5 but requires to look at an even more narrow 
class of functions for which the Betti vector 
at a critical point should only change at one entry 
$b_m$ paraphrasing the addition or removal of a $m$-dimensional ``handle". While Morse theory needs
assumptions on functions, the Lusternik-Schnirelmann theorem works
for general finite simple graphs.
Discrete Morse theory has been pioneered in a different way by Forman 
since 1995 (see e.g. \cite{Forman1999,Forman2002}) and proven extremely useful.
Forman's approach also builds on Whitehead and who looks at classes 
of functions on simplicial complexes for which strong Morse 
inequalities (and much more) was proven. This theory is much more developed than the pure
graph theoretical approach persued here.  \\

To get the lower bound $\ccup(G) \leq \tcat(G)$ for category implying $\ccup(G) \leq \cat(G)$,  
we need a cohomology ring, that is an {\bf exterior multiplication} on {\bf discrete differential forms} 
and an {\bf exterior derivative}. To get a Grassmannian algebra, to define the cup length, a homotopy invariant. 
Let $\G_k$ denote the set of $K_{k+1}$ subgraphs of the finite simple graph $G=(V,E)$ so that $\G_1=V,\G_2=E$. 
If $v_k$ is the cardinality of $\G_k$, then $\Omega_k$, the set of all anti-symmetric functions in $k+1$ variables
$f(x_0,x_1, \dots, x_k)=f(x_0,\vec{x})$ forms a vector space of dimension $v_k$.
The linear space $\Omega = \oplus \Omega_k$ of discrete differential forms has dimension $v=\sum_k v_k$ and the super trace
of the identity map is the Euler characteristic of $G$. Given a $p$ form $f$ and a $q$ form, 
we can define the tensor product $f \otimes g(x_0,\vec{x},\vec{y}) = f(x_0,\vec{x}) g(x_0,\vec{y})$ of $f$ and $g$ 
centered at $x_0$. Define now
$f \wedge' g(x_0,\vec{x},\vec{y}) = \sum_{(\vec{u},\vec{v})=\sigma(\vec{x},\vec{y})} (-1)^{\sigma} f(x_0,\vec{u}) g(x_0,\vec{v})$,
where $\sigma$ runs over all $p,q$ shuffles, permutations of $\{1,\dots,p+q \; \}$ which preserve the order on the first $p$
as well as the last $q$ elements. Finally define $f \wedge g$ as ${\rm Alt}(f \wedge' g)$, the antisymmetrization
over all $p+q+1$ elements. Except for the last step which is necessary because different ``tangent spaces" come together at
a simplex, the definitions are very close to the continuum and work for general finite simple
graphs even so the neighborhoods of different vertices can look very different. The exterior algebra $(\Omega,\wedge)$ is an associative
graded algebra which satisfies the super-anti-commutativity relation $f \wedge g =(-1)^{|f| |g|} g \wedge f$ if
$p=|f|$ denotes the order of the form $f$. The exterior derivative $d:\Omega_p \to \Omega_{p+1}$ defined by 
$df(x_0,x_1,\dots,x_p) = \sum_{k=0}^p (-1)^k f(x_0,x_1,\dots, \hat{x}_k,\dots,x_p)$ leads to $H^p(G)={\rm ker}(d_p)/{\rm im}(d_{p-1})$
in  the same as in topology. The dimension $\beta_p$ of the vector space $H^p(G)$ is a Betti number and linear algebra
assures the Euler-Poincar\'e formula $\chi(G)=\sum_{p} (-1)^p \beta_p = \sum_p (-1)^p v_p$. In summary, for any finite simple 
graph $G$, we have a natural associative, graded super differential algebra $(\Omega,\wedge,d)$ which satisfies the Leibniz rule
$d(f \wedge g) = df \wedge g + (-1)^{|f|} f \wedge dg$ and consequently induces a cup product and so a cohomology ring 
$(H^*(G),\wedge)$. \\

The estimates $\ccup(G) \leq \tcat(G) \leq \crit(G)$ link as in the continuum 
an algebraically defined number $\ccup(G)$ with the topologically defined number $\tcat(G)$ and 
the analytically defined number $\crit(G)$. We can make $\tcat$ and $\crit$ homotopy invariant similarly
as $\gcat$ was made homotopy invariant. We have then $\ccup(G) \leq \cat(G) \leq \cri(G)$, relating three
homotopy invariants. This works unconditionally for any finite simple network $G$. \\

Updates since Nov 4: \\
\begin{itemize}
\item Nov. 6, 2012: Scott Scoville informed us on \cite{AaronsonScoville} with Seth Aaronson,
where a discrete LS category to the number of critical points in the sense of Forman. There seems no
overlap with our results.
\item Nov. 13, 2012: 
\begin{itemize}
\item We renamed the original category $\tcat$ and define $\cat(G) = \min_{H} \tcat(G)$, where $H$ runs over all 
graphs homotopic to $G$. Now only, $\cat$ is a homotopy invariant. The topological category $\tcat$, like $\gcat$
fails to be an invariant. Theorem 1 with $\tcat(G) \leq \crit(G)$ remains the same.
\item We introduce the homotopy invariant 
$\cri(G)=\min_{H} \crit(G)$, where $H$ runs over graphs homotopic to $G$.  
The corollary $\ccup(G) \leq \cat(G) \leq \cri(G)$ relates now three homotopy invariants. 
(3 letter words like $\cat$ now indicate homotopy invariants while 
4 letter words like $\gcat$ indicate quantities which are not homotopy invariant.)
For the dunce hat, we have $\tcat(G)=2,\crit(G)=3$ but $\cat(G)=1,\cri(G)=1$.
\item We add reference \cite{Grady} which deals with discrete forms but does not introduce an exterior
algebra formalism.
\item The new reference \cite{BrachoMonteiano} brings in the "skorpion" 
an example with $\ccup(G)=2, \cat(G)=2$ and $\tcat(G)=3$. 
\item The right part of Figure 6 is now completely triangulated and matches now the graph to the left. 
\item Lemma 3 is now before theorem 1. This helps to make the proof more clear.
\end{itemize}
\end{itemize}

\section{Critical points and Euler characteristic}

In this section, we see a new proof of the Poincar\'e-Hopf theorem \cite{poincarehopf} telling that the 
sum of the indices over all critical points is the Euler characteristic.  \\

Assume $G=(V,E)$ is a finite simple graph, where $V$ denotes the vertex set and $E$ the edge set. 
A subset $W$ of $V$ {\bf generates} a subgraph $(W,F=\{ (a,b) \; | \; a \in W, b \in W \; \})$ of $G$. 
For a vertex $x \in V$, denote by $S(x)$ the {\bf unit sphere}, the subgraph of 
$G$ generated by the vertices connected to $x$.  For an injective function $f:V \to \R$, 
denote by $S^-_f(x)$ the subgraph of $G$ generated by the vertices 
$\{ y \in S(x) \; | \; f(y) < f(x) \; \}$.

\begin{defn}
If $v_k$ is the number of complete $K_{k+1}$ subgraphs of $G$, 
then the Euler characteristic $\chi(G)$ 
of $G$ is defined as the finite sum $\chi(G) = \sum_{k=0}^{\infty} (-1)^k v_k$.
\end{defn}

\begin{defn}
A point $v \in V$ is a {\bf critical point with nonzero index} of $f$ if the index 
$i_f^-(x) = 1-\chi(S^-_f(x))$ is nonzero.
\end{defn}

{\bf Examples.} \\
{\bf 1)} Any minimum of $f$ is a critical point of $f$ with index $1$. \\
{\bf 2)} A maximum is a nonzero index critical point if $S(x)$ 
has Euler characteristic different from $1$.
This happens for example on cyclic graphs $C_n$, where maxima have index $-1$. \\
{\bf 3)} For an icosahedron, a sphere like graph, every maximum and minimum is a critical point of index $1$ 
because $S^-(x)=\emptyset$ at a minimum and $S^-(x)$ is a cyclic graph of Euler
characteristic $0$ if $x$ is a maximum for $f$.  \\

\begin{lemma}
Given two subgraphs $G_1=(V_1,E_1),G_2=(V_2,E_2)$ of $G$ whose union is 
$G=(V=V_1 \cup V_2,E=E_2 \cup E_2)$ 
and which have have intersection $H=(V_1 \cap V_2,E_1 \cap E_2)$. Then
$$ \chi(G) = \chi(A) + \chi(B) - \chi(H) \; . $$
\end{lemma}
\begin{proof}
The number $v_k$ of $K_{k+1}$ subgraphs of $G_1,G_2,G,H$ satisfies
$$ v_k(G) = v_k(G_1) + v_k(G_2) - v_k(H) \; . $$
Adding up the alternate sum, leads to the claim.
\end{proof}

Given an injective function $f:V \to \R$ 
and a vertex $x \in V$, define $G_f(x)$ as the subgraph of $G$ which is
generated by the vertices $V(x) = \{ y \; | \; f(y)<f(x) \; \}$.
Nonzero index critical points are the vertices,
where the Euler characteristic of $G(x)=(V(x),E(x))$ changes:

\begin{propo}[Poincar\'e-Hopf]
$\chi(G(x)) = \sum_{y \in V(x)} i_f(y)$ for any injective function $f$ and vertex $x$. 
\end{propo}

\begin{proof}
By the lemma, a vertex $x$ changes the
Euler characteristic from $\chi(\{x\} \cup S^-(x)) = 1$ to $\chi(S^-(x)) \neq 1$.
Adding up the indices gives so $\chi(G(a)) = \sum_{x \in V(a)} i_f(x)$.
\end{proof}

This provides an alternative proof for the Poincar\'e-Hopf theorem 
$$  \chi(G) = \sum_{x \in V} i_f(x) \;  $$
for graphs (see \cite{indexexpectation}).

\section{Contractibility}

Contractibility in itself is defined inductively with respect to the order $|V|=v_0$ of the graph.

\begin{defn}
A graph with one vertex is {\bf contractible in itself}. Having defined contractible in itself 
for graphs of order $|V|<n$, a graph $G=(V,E)$ of order $|V|=n$ is called {\bf contractible in itself}, 
if there exists an injective function $f:V \to \R$ 
such that $S^-_f(x)$ is contractible in itself or empty for every vertex $x$. 
\end{defn}

{\bf Remark}. Since $S^-_f(x)$ does not contain the point $x$, the order of $S^-(x)$ is smaller
than the order of $G$ and the induction with respect to the order is justified. 

\begin{defn}
Given a graph $G=(V,E)$ and an additional vertex $x \in V$, the 
new graph $G'=(V \cup \{ x \},E \cup \{ (y,x) \; | \; y \in V \; \})$
is called the {\bf pyramid extension} of $G$. The new vertex $x$ has the unit sphere 
$S(x)=G$. 
\end{defn}

This construction allows to build up contractible graphs by aggregation. 
If $G$ is contractible then the new graph $G'$ is still contractible. 

\begin{propo}
Every in itself contractible graph can be constructed from a single point graph by
successive pyramid extension steps.
\end{propo}
\begin{proof}
By definition, there exists an injective function $f$ on $G$ which has only the
global minimum $z$ as a critical point. The sequence of graphs $G_f(a)$ defined by
this function $f$ produces pyramid extensions. Because $f$ is injective, only one vertex
is added at one step. In each case, the extension uses $H=S^-_f(x)$ as the contractible
set at the pyramid extension is done.
The reverse is true too. A sequence of pyramid constructions defines an
injective function. Every additional vertex has a function value larger than all the other 
function values. 
\end{proof}

\begin{defn}
The deformation $G \to G'$ is called a {\bf homotopy step} if it is 
a pyramid extension on a contractible subgraph $H$ or if it is the 
reversed process where a point $x$ for which $S(x)$ is contractible of $G'$ together with all connections 
is removed. Two graphs are called {\bf homotopic} if one can get from one to the other by a 
finite sequence of homotopy steps. 
\end{defn}

\begin{propo}
Homotopy defines an equivalence relation on finite simple graphs.
\end{propo}
\begin{proof}
A graph is homotopic to itself. If $G$ is homotopic to $H$ and $H$ is homotopic to $K$, then 
the deformation steps can be combined to get a homotopy from $G$ to $K$. 
Also by definition is that if $G$ is homotopic to $H$ then $H$ is homotopic to $G$. 
\end{proof}

{\bf Remarks.} \\
{\bf 1)} Graphs can be homotopic to a single point graph without being contractible. \\
{\bf 2)} Using contractions or expansions of a graph alone defines {\bf simple homotopy} between two graphs $G_1,G_2$.
It also defines an equivalence relation on finite simple graphs. As the textbook \cite{Giblin} 
example of the dunce hat example shows, it is a finer equivalence relation leading to more equivalence classes. \\
{\bf 2)} For small $n$, we can look at all connected graphs of order $n$ and count how many homotopy types
there are. The number $h(n)$ of homotopy types does not decrease with $n$ 
since we can also make a pyramid extension over a single point without changing homotopy. While
the number of connected graphs of order $n$ grows super exponentially, the number of homotopy
types might grow polynomially only, but we do not know. So far, no upper bound except the trivial 
$h(n) \leq 2^{n(n-1)/2}$ bound have been established nor a lower bound beside the trivial $h(n) \geq C n$ 
obtained by wedge gluing spheres. But estimating the number $h(n)$ seems also never have been asked
for Ivashchenko homotopy. \\

{\bf Examples.} \\
{\bf 1)} For connected graphs of order 1,2,3, there is exactly one homotopy type. \\
{\bf 2)} For connected graphs of order 4,5, there are 2 homotopy types, $C_4,C_5$ are not contractible. \\
{\bf 3)} For connected graphs of order 6,7, there are 4 homotopy types since we can already build the 
octahedron, a two-dimensional non-contractible sphere, as well as a figure $8$ graph.  \\
{\bf 4)} For connected graphs of order 8, there are already at least 7 homotopy types, 
since we can build a 3 dimensional sphere,
the 16 cell as well as a circle with two ears, a clover both with Euler characteristic $-2$ and an 
octahedron with a handle. Figure~\ref{homotopyclass} shows $10$ classes. The last four are all homotopic
and algebraic topology can not distinguish them. The Betti vectors (the length of the vector 
is always the dimension of the maximal simplex in the graph) 
for the 10 graphs in the order given are $\beta=(1,1)$, $\beta=(1,0,0,0,0,0,0,0)$,
$\beta=(1,0,1,0)$, $\beta=(1,2)$, $\beta=(1,0,0,1)$, $\beta=(1,1,1)$,
$\beta=(1,3)$, $\beta=(1,3)$, $\beta=(1,3)$, $\beta=(1,3)$.  \\

\begin{figure}
\scalebox{0.25}{\includegraphics{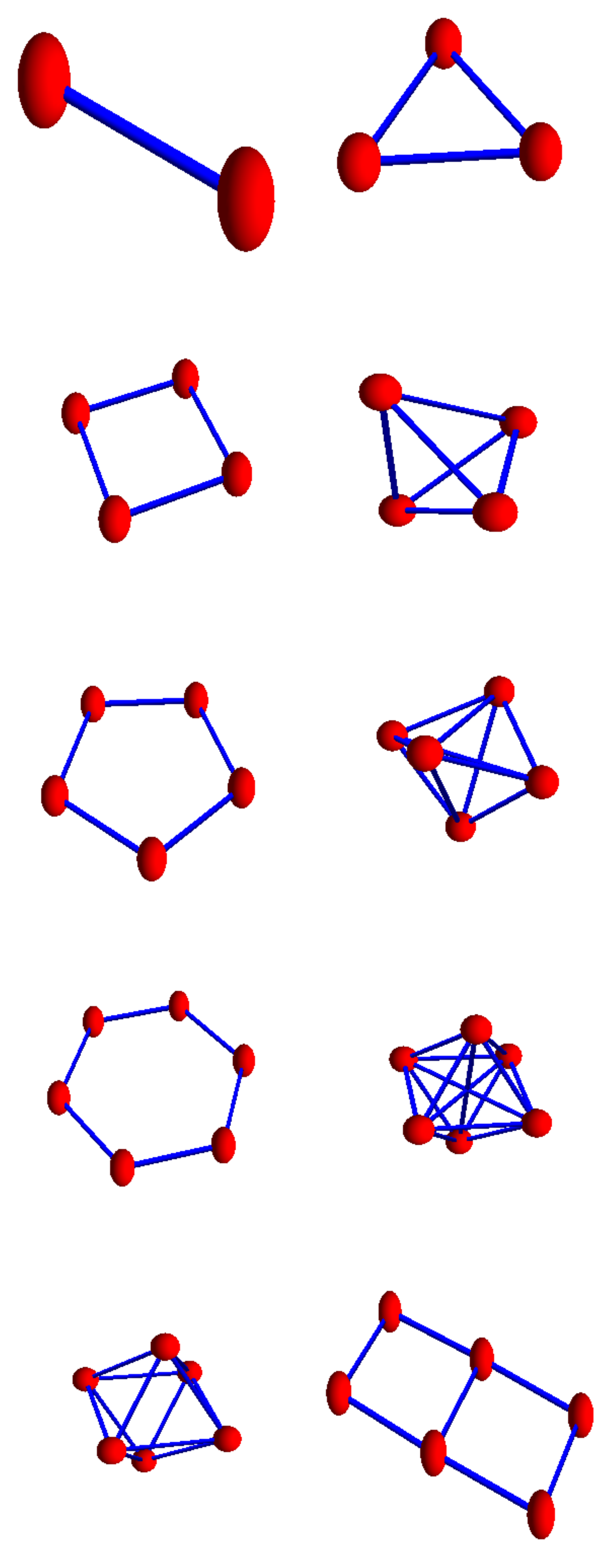}}
\scalebox{0.25}{\includegraphics{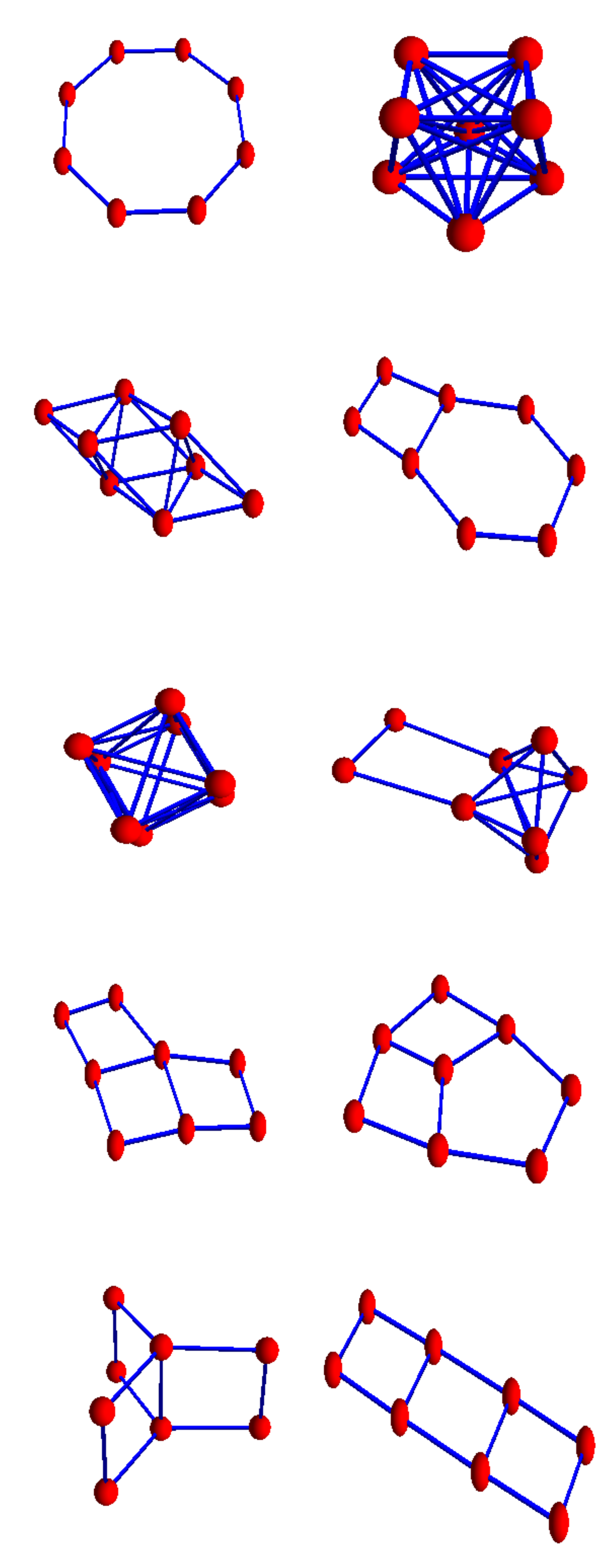}}
\caption{
Homotopy classes of graphs for small order. To the left, we see homotopy types for
graphs of order up to $6$. To the right, we see examples of 
homotopy types for graphs of order $8$. The picture illustrates
$h(1)=h(2)=h(3)=1, h(4)=h(5)=2,h(6)=h(7)=4$ and that $h(8) \geq 7$.
The last $4$ graphs in the picture have all the same Betti vector 
$\beta=(1,3)$ so that they are homologically indistinguishable. Indeed, one 
can deform them into each other.
A computer search is in general expensive since there are $2^{n(n-1)/2}$ graphs of order 
$n$ which is $35184372088832$ for $n=9$ already and a brute force check for 
homotopy would need many comparisons
\label{homotopyclass}
}
\end{figure}

{\bf Remarks.} \\
{\bf 1)} The homotopy definition given in \cite{I94} uses also
edge removals and additions which can be realized with pyramid extensions  \cite{CYY}.
This notion of homotopy for graphs goes back to deformation notions defined in \cite{Whitehead,Alexander}. \\
{\bf 2)} As in the continuum, homotopy is not inherited by subgraphs. If $G$ is homotopic to $K$
and $H$ is a subgraph of $G$, then $H$ does not necessarily have a homotopic deformation in $K$.
A cyclic subgraph $H$ of a complete graph $G=K_5$ for example can not pass over to a homotopic analogue 
in a one point graph $K=P_1$. \\

\section{Contractibility within a graph}

Assume we have a homotopy deformation $G \to G'$ from a graph $G$ to second graph $G'$. 
This deformation is given by contraction or expansion steps 
$$  G=G_0,G_1,G_2,\dots,G_n=G' \; . $$ 
Assume now we have a subgraph $K$ of $G$. For any homotopy step $G_k \to G_{k+1}$ over a subgraph $H$ 
and a subgraph $K_k$ of $G_k$, we can make the pyramid extension of $H_k = H \cap K_k$. 
We can follow $K=K_0,K_1, \dots, K_n=K'$ with this deformation, where $K_{k+1}$ is obtained
from $K_k$ by making a pyramid extension over $H_k$
Because $H_k$ are not necessarily contractible in themselves, these graphs are not 
necessarily homotopic. Different components can merge or holes can appear and disappear. 

\begin{defn}
Given a homotopy deformation $G \to G'$ and a subgraph $K$ of $G$, we call the deformed graph $K'$
the {\bf deformation induced by the homotopy}.
\end{defn}

{\bf Remark.} The deformed graph $K'$ not only depends on $K$ and $G$, it also 
depends on the chosen deformation. 

\begin{defn}
A subgraph $H$ of $G$ is called {\bf contractible in $G$} if it is contractible in itself
or if there is a contraction $G'$ of $G$
for which the induced deformation $H'$ is contracted to a graph with one vertex only. 
\end{defn}

{\bf Remark.} \\
{\bf 1)} One could introduce a notion {\bf $H$ is homotopic to point within $G$}
by asking that there is a homotopy deformation (and not necessarily only a contraction) 
$G'$ to $G$ such that the induced deformation $H'$ has one vertex only. Using this instead
of contractible in $G$ would change things.  The Bing house, the igloo or the dunce hat show
that this is not the same. The dunce hat would have category $1$ with this notion. \\
{\bf 2)} In order that a subgraph of $G$ is contractible in $G$ to a point, we 
do not allow general homotopy deformations of $G$. 
This is also always the assumption in the continuum, where contractibility of a subset $H$ in $M$ 
is defined that the inclusion map $H \to M$ is homotopic to a constant map. \\
{\bf 3)} The notion contractibility in itself has to be added in ``contractibility in G'':
for example, since there are no contractions of
the octahedron, only the one-point subgraphs are contractible within $G$ in the narrower sense
so that the category of the octahedron would be $6$, the number of vertices. This is
larger than the minimal number $2$ of critical points. \\

{\bf Examples.} \\
{\bf 1)} The entire graph $G$ is contractible in itself if and only if it is contractible in $G$. \\
{\bf 2)} Any subgraph of a contractible graph $G$ is contractible in $G$. \\
{\bf 3)} In the continuum, if $G$ is connected, then a finite discrete subgraph with 
         no vertices is contractible. One can deform the graph such that the points 
         become a single point. This is not always true in the discrete.
         The graph $P_4$ in $C_4$ for example can not be deformed to a point simply 
         because no contractions of $C_4$ are possible.\\
{\bf 4)} One could call $G$ {\bf simply connected} if every cyclic 
         subgraph in $G$ is contractible in $G$ but this would be
         too narrow. 
         The equator $H$ in an octahedron $G$ for example is not contractible in $G$ because
         no contractions of $G$ are possible. There are homotopy deformations of 
         $G$ however which contract $H$ to a point: blow up the octahedron to an 
         icosahedron for example using a cobordism with a three dimensional graph having the 
         octahedron and icosahedron as a boundary, then shrink the icosahedron back, 
         this time however deforming the circle to a point.  \\

Both ``contractibility in itself'' and ``contractibility in $G$'' are not homotopy invariants.
There are graphs which are contractible in $G$ but where a contraction renders it non-contractible.

\section{Category} 

\begin{defn}
A finite set of subgraphs $G_j=(V_j,E_j)$ of $G=(V,E)$ is a
{\bf cover} of $G$ if $\bigcup_j V_j=V$ and $\bigcup_j E_j=E$.
\end{defn}

\begin{defn}
The {\bf topological Lusternik-Schnirelmann category} 
$\tcat(G)$ is the minimal number $n$ for which
there is a cover $\{U_j\}_{j=1}^n$ of $G$ with subgraphs $U_j$ of $G$ 
such that each $U_j$ is contractible in $G$.
Such a cover $\{U_j \; \}_{j=1}^n$ is called a {\bf category cover}.
The {\bf Lusternik-Schnirelmann category} $\cat(G)$ is defined as
the minimum of $\tcat(H)$, where $H$ runs over all to $G$ homotopic graphs.
\end{defn}

By definition, category is a homotopy invariant. 
To illustrate the contrast to contractibility, lets look at 
geometric Lusternik-Schnirelmann category as it is done in the continuum. 

\begin{defn}
The {\bf geometric Lusternik-Schnirelmann category} $\gcat(G)$ 
is the minimal number of {\bf in themselves contractible subgraphs} of $G$ which cover $G$.
The {\bf strong category} $\Cat(G)$ as the minimum of geometric categories 
$\gcat(G)$ among all $G'$ which are homotopic go $G$.
\end{defn}

{\bf Remarks.} \\
{\bf 1)} The geometric Lusternik-Schnirelmann category is not a homotopy invariant. 
This is the same as in the continuum. The strong category is an invariant. It has been introduced
in \cite{Ganea}. The dunce hat $G$ shows that also the topological Lusternik-Schnirelmann category 
is not a homotopy invariant. The graph $G$ is not collapsible in $G$ but there are graphs homotopic to
$G$ which have $\tcat(G)=1$. \\
{\bf 2)} The counting for category differs in the literature. We follow the counting assumption 
from papers like \cite{Fox,Takens,Bott82,Clapp,Singhof} which assume the category of a 
single point is $1$, which leads for manifolds to category values
$\cat(SU(n))=n$, $\cat(S^n)=2$, $\cat(RP^n)=n+1$, $\cat(T^n)=n+1$. 
An other part of the literature like \cite{CLOT} define $\cat(G)$ to be 
by one smaller which is more convenient in other situations. 
We chose the former definition because we want category in general to agree with the 
number of critical points as well as with the cohomologically defined cup length. \\
{\bf 3)} Since for category one has more possible sets to choose from, one 
has $\cat(G) \leq \Cat(G)$ and in general no equality as the example of Fox \cite{Fox} 
illustrated in Figure~{\ref{fox}} shows. \\
{\bf 4)} Again we see here why we have to require ``contractibility in $G$ ''.
The octahedron $G$ has the strong category $\Cat(G)=2$ and
would have $\tilde{\cat}(G)=6$, if $\tilde{\cat}(G)$ were the category defined by
``contractibility in $G$'' in the narrow sense without allowing the elements of the 
cover to be contractible. \\
{\bf 5)} $\gcat(M)-\cat(M)$ can be arbitrary large \cite{ClappMontejano} in the continuum.
By using triangularizations, we expect $\gcat(G)-\tcat(G)$ to become arbitrary large
also in the graph theoretical sense. \\
{\bf 5)} A theorem of Ganea tells that 
the strong category satisfies $\cat(M) \leq \Cat(M) \leq \cat(M)+1$ in the continuum. Again,
this is expected to be the same in the graph theoretical sense. \\
{\bf 6)} A theorem of Singhof tells $\cat(M \times T^1) = \cat(M) + 1$ in the continuum if the
space is nice enough and the dimension is large enough. 
A conjecture of Ganea stated $\cat(M \times S^n) = \cat(M)+1$ in general
but there are counter examples even for smooth manifolds by Iwase who found also 
smooth manifold $M$ counter examples to $\cat(M \setminus \{p\}) = \cat(M)-1$.  \\

\begin{figure}
\scalebox{0.35}{\includegraphics{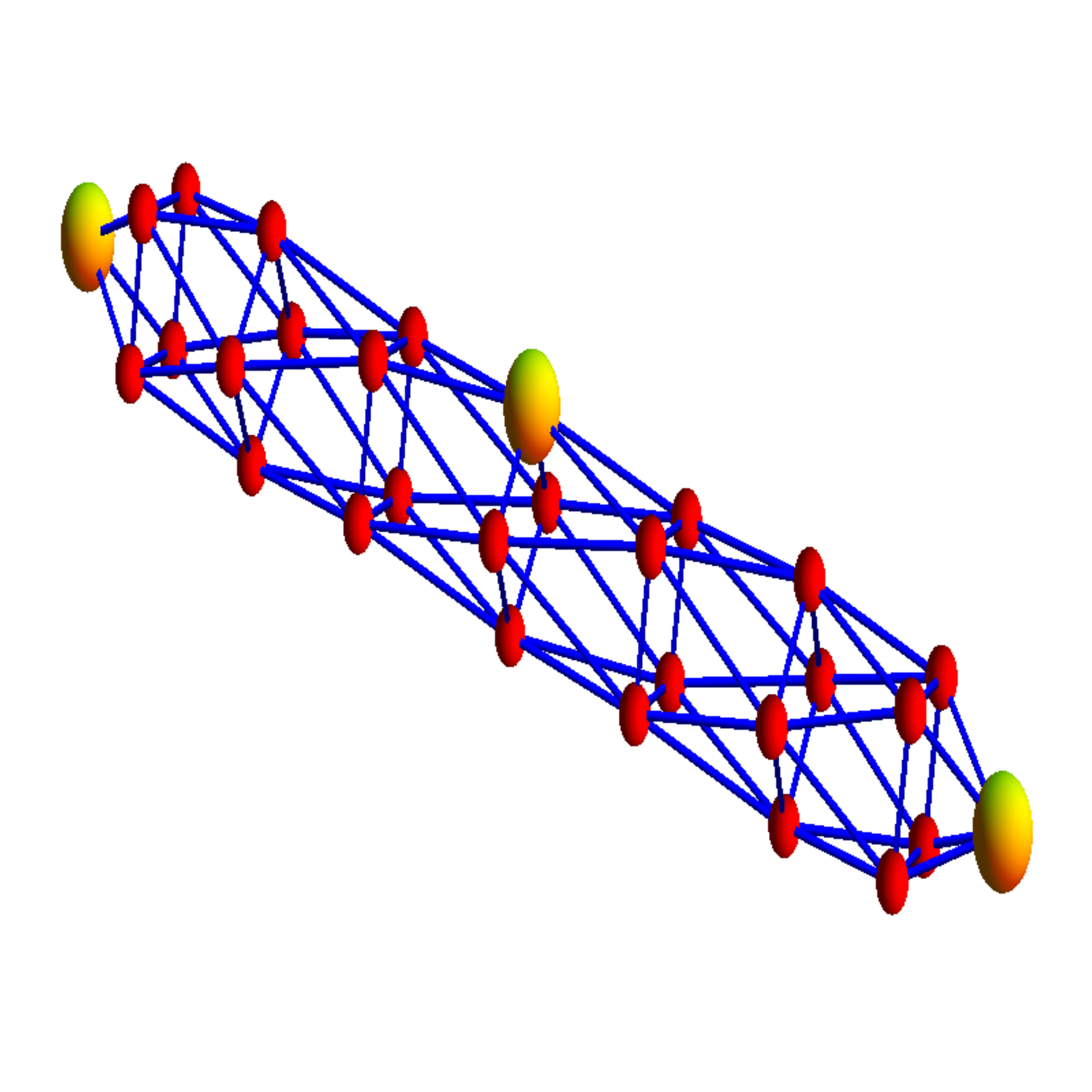}}
\caption{
A graph version of the example of Fox \cite{Fox} is
a spherical polyhedron, where
three vertices are identified. It is possible to cover the 
graph $G$ by two subgraphs $U,V$ which are in $G$ contractible so that 
$\tcat(G)=2$. However, the geometric category is $\Cat(G)=3$. We need three
subgraphs contractible in themselves to cover $G$. It is an example showing
why category is defined as it is. An other example is the octahedron which 
is noncontractible and illustrating that in the discrete, we need ``contractibility in $G$''
as defined. In the continuum, this is
not necessary if the topological space $M$ is normal
and the subsets $U_j$ of $M$ appearing in the cover are neighborhood retracts of $M$. 
\label{fox}
}
\end{figure}

{\bf Examples.}  \\
{\bf 1)} The discrete graph $G=P_n$ of $n$ vertices and no edges has $\cat(G)=\Cat(G)=n$. \\
{\bf 2)} The graph $G=K_n$ has category $\cat(G)=\Cat(G)=1$ because it is contractible. \\
{\bf 3)} Any connected tree $G$ has category $\cat(G)=\Cat(G)=1$. \\
{\bf 4)} The circular graph $C_n$ has $\cat(C_n)=\tcat(C_n)=\Cat(C_n)=2$.  \\
{\bf 5)} The octahedron and the icosahedron both have $\cat(G)=\tcat(G) = \Cat(G)=2$. \\
{\bf 6)} Any discrete two torus graph satisfies  $\cat(G)=\tcat(G)=\Cat(G)=3$. \\
{\bf 7)} A figure 8 graph has $\cat(G) = \tcat(G)=\Cat(G)=2$. \\
{\bf 8)} Figure~\ref{fox} shows an example of a graph where $\cat(G)=2<\Cat(G)=3$.  \\
{\bf 9)} For the dunce hat, $\tcat(G)=2=\Cat(G)$. 

\begin{lemma}
A pyramid extension over the full graph $G$ has category $\cat(G)=\tcat(G)=\Cat(G)=1$. 
\end{lemma}
\begin{proof}
Take a function on the vertex set $V$ which has the minimum at $x$. For every vertex $y \in V$
different from $x$, the graph $S^-(y)$ is a pyramid extension and so contractible, 
independent on whether $G$ was contractible. This means that $x$ is the only critical point. 
\end{proof}

It follows that any unit ball $B(x)$ in a graph always is of category $1$ because
$B(x)$ is the pyramid extension of the sphere $S(x)$. 

\begin{defn}
A vertex $x$ is called a {\bf regular point} for $f$ if $S_f^-(x)$ is contractible in itself. 
Any other point is called a {\bf critical points} of $f$.
\end{defn}

{\bf Remarks.} \\
{\bf 1)} Figure~\ref{example1} shows an example of a critical point with zero index.
It is an example for which $S^-(x)$ has Euler characteristic $1$ even so $S^-(x)$ is not
contractible.  \\
{\bf 2)} A graph $G=(V,E)$ is contractible if and only if there is an injective
function $f: V \to \R$ such that $f$ has only one critical point.
The only critical point is then the global minimum.  \\
{\bf 3)} A contractible graph $H$ is connected because every minimum on a connected component is a critical point. \\
{\bf 4)} A contractible graph has Euler characteristic $1$. \\
{\bf 5)} Unlike in the continuum, only minima and not maxima
are always critical points. In the continuum, the boundary matters. For the open unit disc $M$ in the plane for 
example we can have a function which has only a saddle point as a critical point. 
In the discrete minima are always a critical point. \\
{\bf 6)} A subgraph $H$ of $G$ can be contractible in $G$ without being contractible in itself. \\

\begin{propo}
Homotopic graphs have the same Euler characteristic. Euler characteristic is a homotopy 
invariant. 
\end{propo}
\begin{proof}
Any homotopy deformation step
with a new vertex $x$ over a subgraph $H$ or its reverse leaves the Euler
characteristic invariant. Then 
$\chi(G \cup \{ x \; \}) = \chi(G) + \chi(B(x)) - \chi(H) = \chi(G) + 1 - 1= \chi(G)$. 
\end{proof}

{\bf Remark.} \\
{\bf 1)} This statement would also have followed from the Euler-Poincar\'e formula $\chi(G) = \sum_j (-1)^j b_j$, 
where $b_j$ is the $j$'th Betti number. Since homology groups are homotopy invariants \cite{I94}, also 
Euler characteristic is.  \\

{\bf Examples.} \\
{\bf 1)} By definition, any contractible graph is homotopic to a single point.  \\
{\bf 2)} Two circular graphs $C_n,C_m$ with $n>3,m>3$ are homotop. To show this, we build an annulus shaped
graph which has $C_n,C_m$ as boundary and which is both homotopic to $C_n$ and $C_m$. This example illustrates
how h-cobordism can be seen as a special case of homotopy. \\
{\bf 3)} Similarly, an octahedron and an icosahedron both have category $2$. They are homotopic but
not in an obvious way. One can not remove a vertex in any of the graphs for example
without changing both category and Euler characteristic. But one can build a three 
dimensional graph which has these two graphs as boundary and which is homotopic. 
Again, this is a  $h$-cobordism.

\begin{figure}
\scalebox{0.13}{\includegraphics{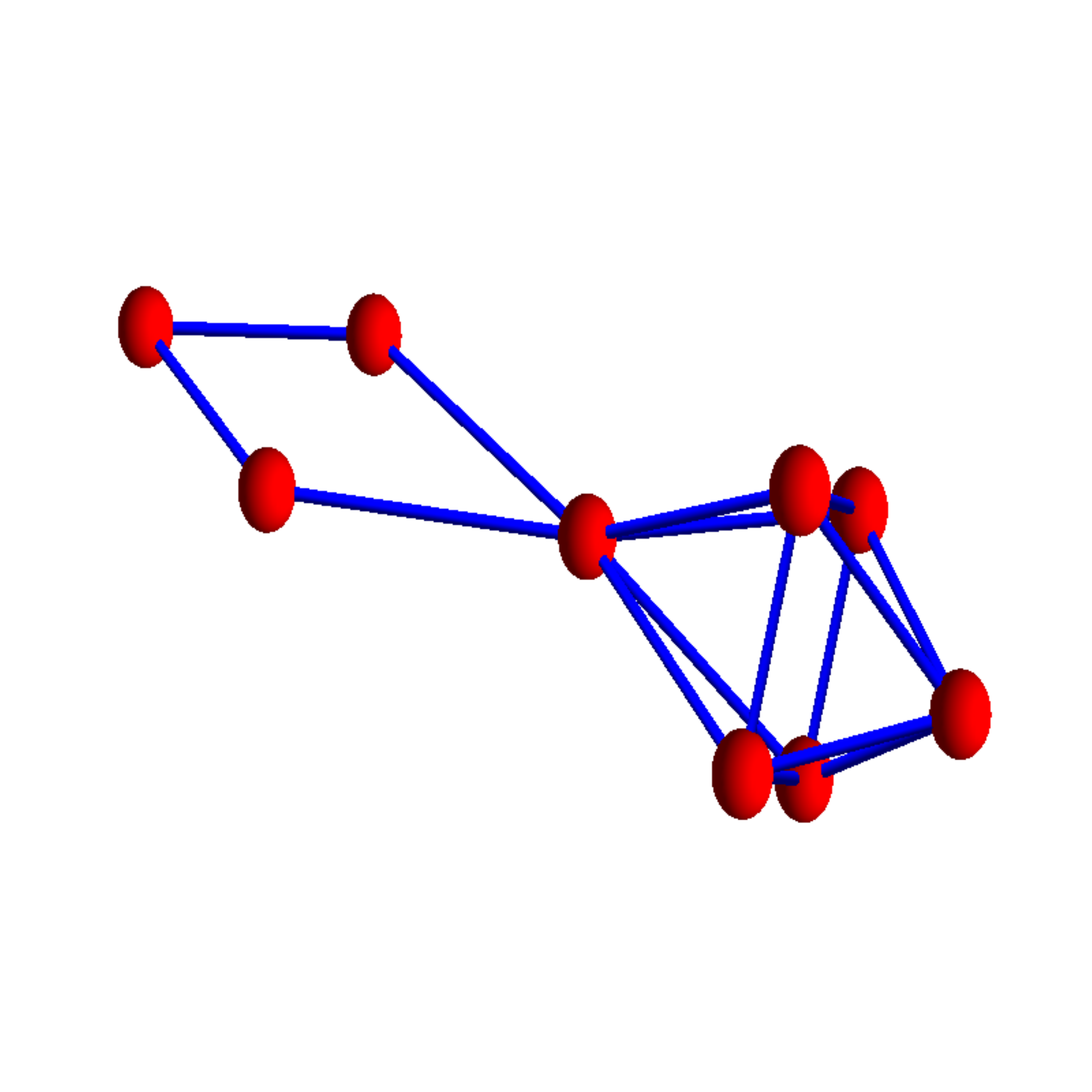}}
\scalebox{0.13}{\includegraphics{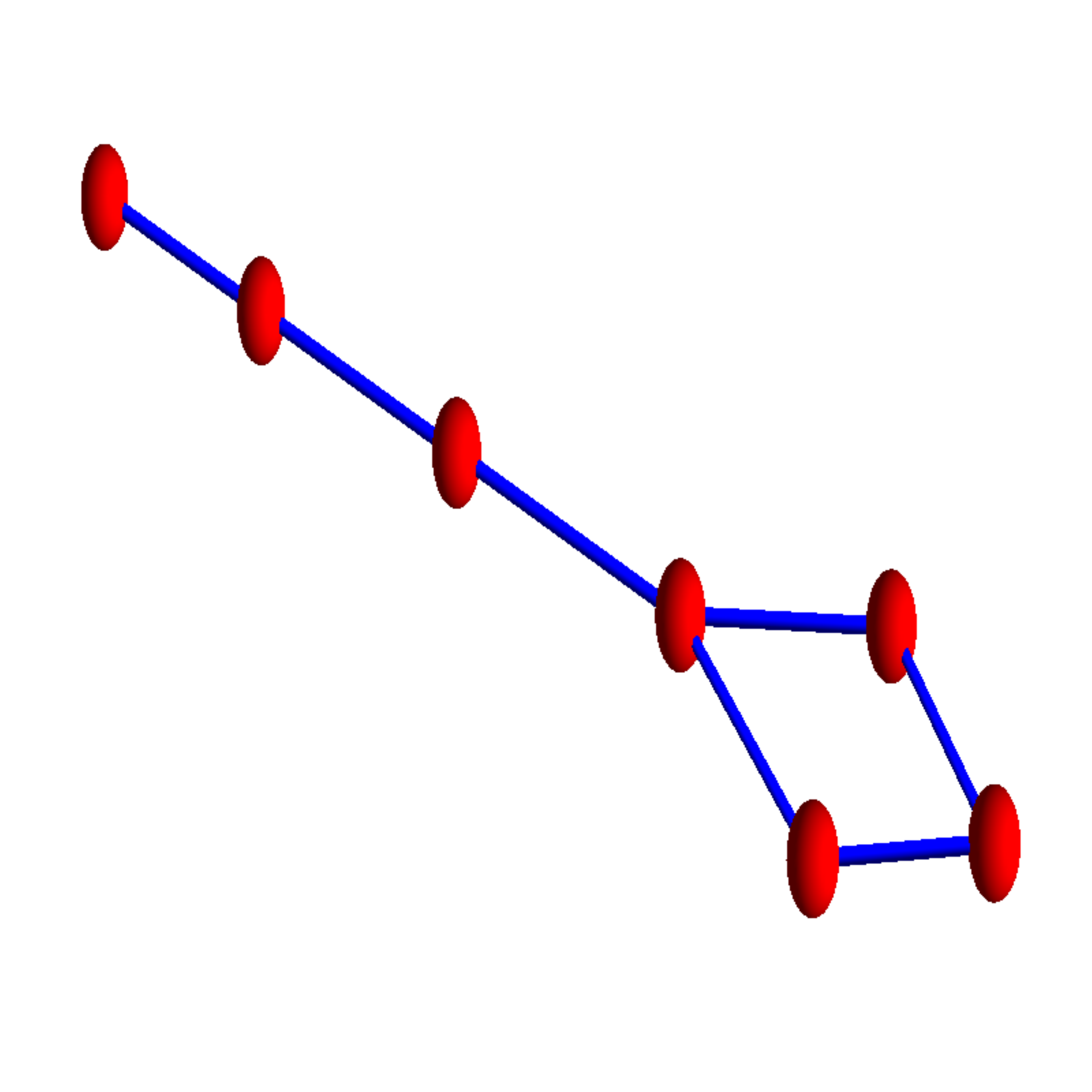}}
\scalebox{0.13}{\includegraphics{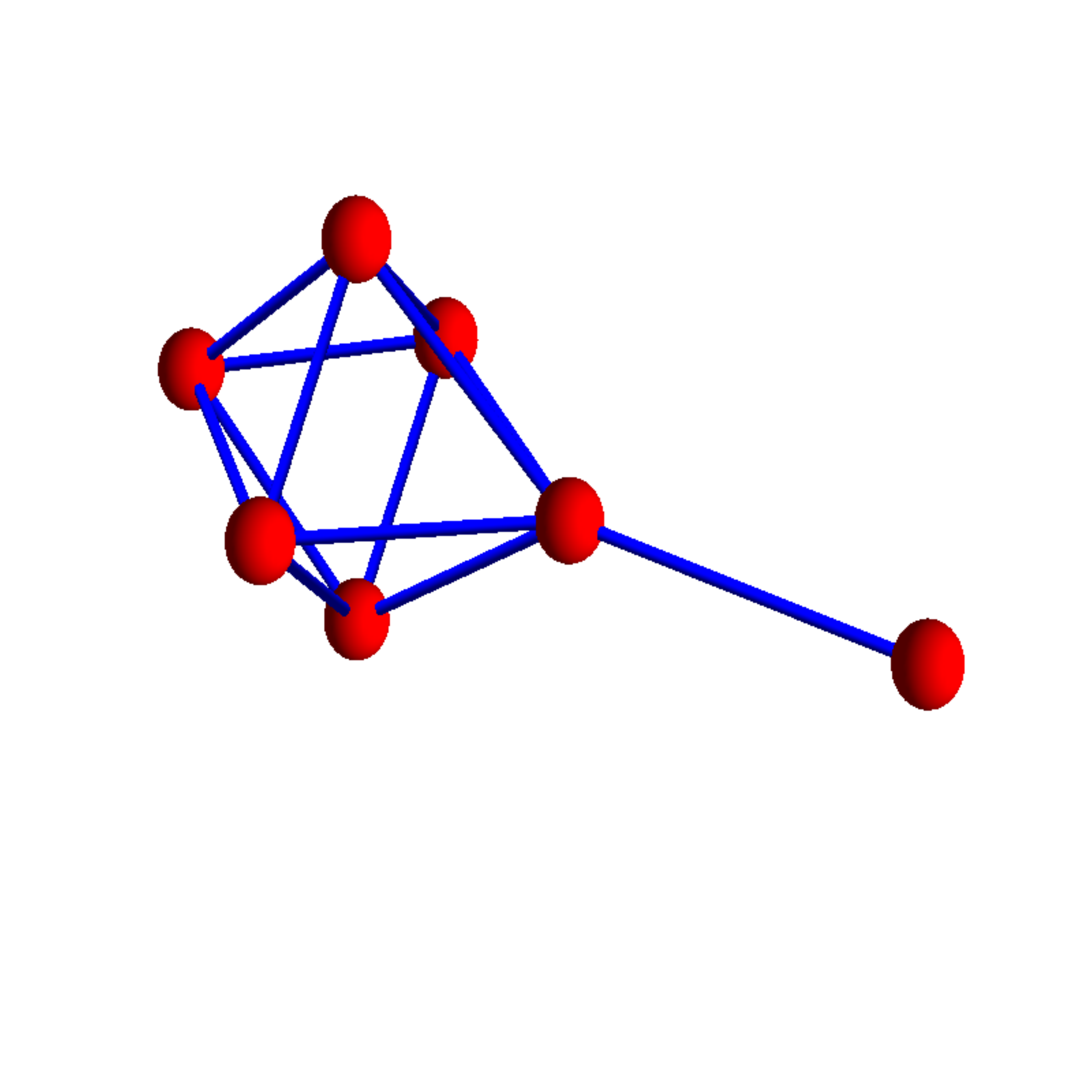}}
\caption{
The left graph is a small example of a graph with a critical point with $i_f(x)=0$.
When adding the vertex $x$, the Euler characteristic of the graph does not change; 
but the category changes from $\cat(H)=3$ to $\cat(G)=1$. 
The middle graph is an example of a critical point for which the category index 
$k_f(x)$ is zero. The right graph provides an example with a critical point for which
both the category index as well as the Poincar\'e-Hopf index are zero. 
Without the link vertex, the category is 2 because there are two contractible 
components, with the link the category is 2 because the graph is homotopic to a 
sphere graph. The Euler characteristic is also 2 in both cases. 
The Betti vector changes from $\vec{b}=(b_0,b_1,b_2)=(2,1,1)$ to $\vec{b}=(1,0,1)$
so that a Morse index would not be defined. 
\label{example1}
}
\end{figure}

\begin{defn}
Call $\crit(G)$ the {\bf minimal number of critical points} which an injective
function $f$ can have on $G$.
\end{defn}

{\bf Remarks.} \\
{\bf 1)} $\crit(G)$ is a simple homotopy invariant because a single homotopy step does not increase the
number of critical points. By reversing it, it also does not decrease the minimal number
of critical points. However, as the dunce hat shows, the minimal number of critical points can 
change if we allow both contractions and expansion. We have to move to $\cri(G)$ as the minimum over
all $\crit(H)$ homotopic to $G$ to get a homotopy invariant. \\
{\bf 2)} One could look at a smaller class of critical points which have the property that $S_f^-(x)$ 
is not ``contractible in $G$'' or empty. This is not a good notion however
because $\gcrit(G)$ is just the number of connected components of $G$.
The reason is that $S^-(x)$ is a subset of $B(x)$ which is contractible in itself
so that $S^-(x)$ is contractible in $B(x)$ and so in $G$. Especially, using the notion ``contractible in $G$''
to define critical points does not lead to a Lusterik-Schnirelmann theorem. Similarly,
using ``contractibility in itself'' instead of 
``contractibility in $G$'' in the definition of the 
category cover would not lead to a theorem. \\

{\bf Examples.} \\
{\bf 1)} A contractible set has $\crit(G)=1$ by definition and so $\cri(G)=1$. \\
{\bf 2)} A discretization of a sphere has $\crit(G)=2$ and $\cri(G)=1$. A discrete Reeb theorem assures
that $\crit(G)=2$ implies that $G$ is homotopic to a discretization of a $k$ dimensional
sphere. Category can see the dimension of the sphere in that the dimension is by $1$
larger than the dimension of $U_1 \cap U_2$ whenever $U_1,U_2$ is a category cover of $G$.
For a cyclic graph for example, $U_1 \cap U_2$ is a discrete 2 point graph which has 
dimension $0$. For a polytop of $\chi(G)=2$ and triangular faces, any category cover $U_1,U_2$
has the property that $U_1 \cap U_2$ is homotopic to a cyclic graph. The dimension is $2$. \\
{\bf 3)} On a discrete two dimensional torus, $\crit(G) = 3$ in general.  \\

\begin{defn}
Call a function $f$ a {\bf Morse function} on $G$ if any vertex is either a regular point
or a critical point of index $1$ or $-1$ and only one cohomology group $H_m(x)$ changes dimension.
The vertex has then the Morse index $m$.  
\end{defn}

This corresponds to the continuum situation, where
at regular points, the half sphere $S^-(x) = S_r(x) \cap \{ f \leq 0 \; \}$ is contractible
and where at critical points, $S^-(x)$  has either Euler characteristic $2$ or $0$ 
corresponding to index $-1$ or $1$.  It follows that like in the continuum, $i_f(x) = (-1)^{m(x)}$
and one has then also the same {\bf strong Morse inequalities} as in the continuum

\begin{propo}
\label{strongmorse}
$$ b_k(G) \leq c_k-c_{k-1}+\cdots \pm c_0, \chi(G) = c_0-c_1+\cdots \; ,  $$
where $c_m$ is the number of critical points of Morse index $m$. 
\end{propo}

\begin{proof}
The proof in \cite{Mil63} section I \$5 works.  
The second equality is a consequence of Poincar\'e-Hopf since the right hand side adds up the indices.
\end{proof}

{\bf Remarks.} \\
{\bf 1)} These inequalities do not hold for a general finite simple graph. See Figure~\ref{example1}
for an example of a critical point with index $0$. \\
{\bf 2)} Geometric graphs like graphs obtained by triangulating a 
manifold admit Morse functions but we do not know yet what the most general class of graphs is
which admit a Morse function. \\
{\bf 3)} Geometric graphs for which every unit sphere is a triangularization 
of a Euclidean sphere have Morse functions.  \\

Critical points for $f$ by definition are vertices $a$, where $G(a)$ changes the number of 
critical points $\crit(G(a))$. There are other points: 
category transition values for $f$ are the values, 
where $G(a)$ changes category and characteristic transition values for $f$
are the values, where $G(a)$ changes Euler characteristic.

\begin{figure}
\scalebox{0.35}{\includegraphics{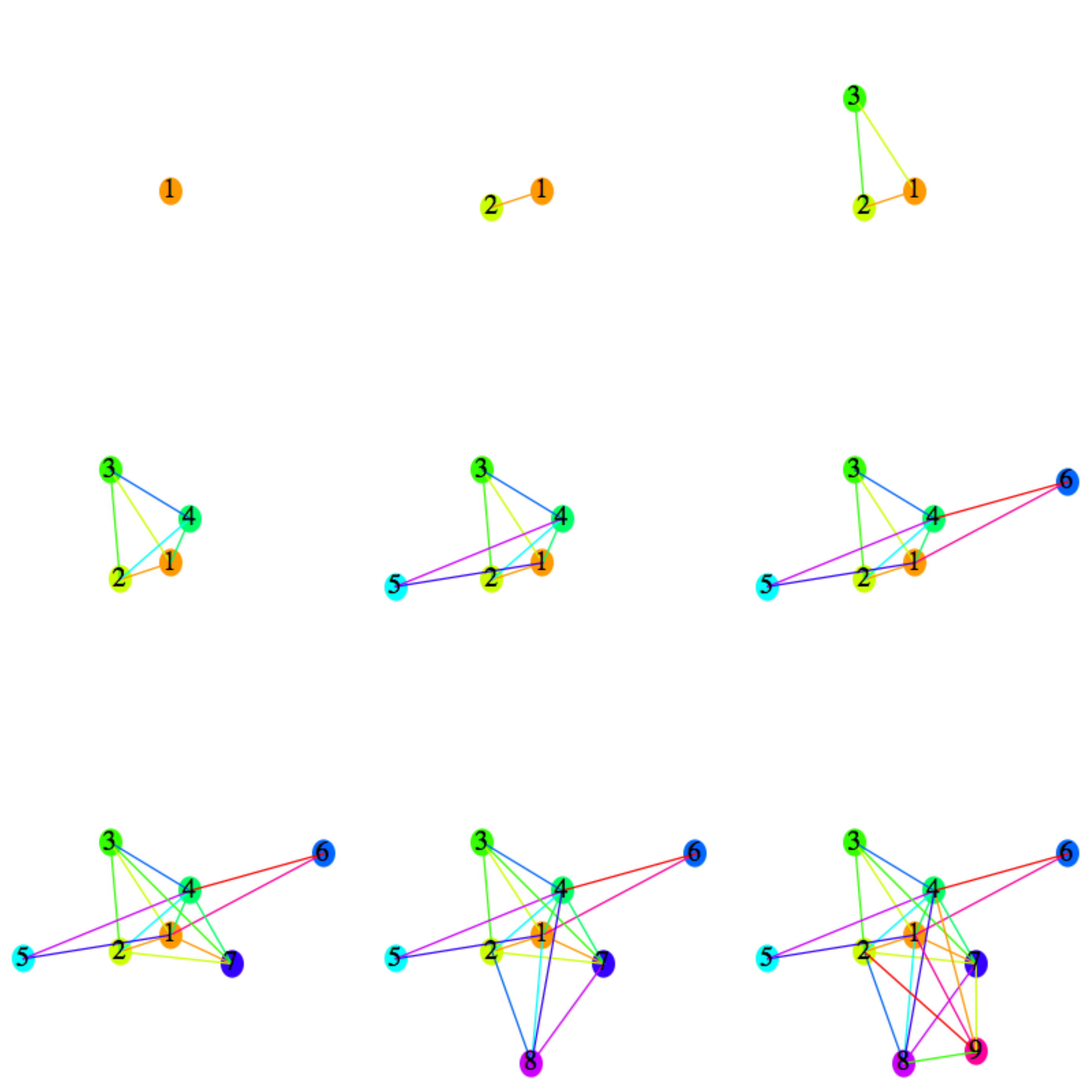}}
\caption{
Ivashkenko homotopy steps define 
a sequence of graphs $G_f(x_j)$ can be seen as a gradient flow
with respect to some function $f$. Time is discrete. Every injective
function $f$ on the vertex set on a simple graph defines such a 
sequence of graphs. The final graph is contractible. Reversing the steps
produces a homotopy to the one point graph. This special homotopy is a 
contraction. 
\label{randomgrowth}
}
\end{figure}

\begin{lemma}
Given a finite simple graph $G$ of topological category
$\tcat(G)=n$ and an in itself contractible subgraph $H$ of $G$. Then there
exists an other category cover $\{ W_j \; \}_{j=1}^n$ of $G$ such that each $W_j \cap H$
is either empty or equal to $H$. The number of contractible elements in the cover remains
the same.
\label{pushing}
\end{lemma}

\begin{proof}
Given a category cover $\{ U_j \; \}_{j=1}^n$ for which
$U_k$ intersects $H$. Expand $U_k$ within $G$
so that it contains the entire graph $H$. This new subgraph $V_k$
is still contractible within $G$.
Now remove contractible components of $U_j \cap H$ from
the other $U_j, j \neq k$ until each of the graphs $W_j, j \neq k$
has no intersection with $H$ anymore. Theses changes
keep $U_j$ contractible within $G$ or contractible in itself. Each change can be
done also so that the complement of $H$ remains covered.
Now $V_k \cap H = H$ is contractible in $G$ and $W_j \cap H$ is empty for $j \neq k$.
\end{proof}

\begin{thm}
$\tcat(G) \leq \crit(G)$. 
\end{thm}

\begin{proof}
To show that a category change of $G(x) \to G(x_{j})$ implies that $x=x_j$ is a critical point,
we show that if we add a regular point $x$, then the category does not change. Note that adding
a point can increase the category only by $1$ so that the number of critical points will be
an upper bound on the category.
We have to check that if $H$ is a contractible subgraph and a pyramid
construction $G'$ is done over $H$ with a vertex $x$, 
then the category of $G$ is the same as the one of $G'$. But the category cover
directly expands too: assume $\{U_j \; \}$ is a category cover of $G$. This means
that $U_j$ is contractible in $G$ or contractible in itself.  By the lemma, we can 
assume that one $U_k$ has the property that $U_k \cap H = H$ and $U_j \cap H = \emptyset$
for $j \neq k$. The pyramid extensions $V_k = U_k \cup [(U_k  \cap H)\cup \{x\}]$ remains
contractible in $G$ or contractible in itself: if $U_k$ was contractible in 
$G$, then $V_k$ is contractible in $G$. If $U_k$ was contractible in itself, then 
$V_k$ is contractible in itself. 
\end{proof}

\begin{coro}
$\cat(G) \leq \cri(G)$. 
\end{coro}
\begin{proof}
Since by definition $\cat(G) \leq \tcat(G)$ we have 
$\cat(G) \leq \crit(G)$. Since the left hand side is a homotopy invariant
we have $\cat(G)=\cat(H) \leq \crit(H)$ for any graph $H$ homotopic to $G$. 
Therefore $\cat(G) \leq \cri(G)$. 
\end{proof}

{\bf Remarks.} \\
{\bf 1)} It can happen that a member $U_j$ of the category cover is contractible in $G$ but
does not remain so after expanding $G$. An example is the wheel graph $W_4$ which is contractible
and contractible in $G$. If we make a pyramid extension over the boundary $S_4$, we get an octahedron $G'$. 
Of course this expansion has added the second critical point.
Now $W_4$ is a subgraph of $G'$ but it is no more contractible in $G'$ because there is no contraction
of $G'$ any more. But it remains contractible in itself. \\
{\bf 2)} A historical remark from the introduction in \cite{RudyakSchlenk}:
the Lusternik-Schnirelmann theorem is contained in the fundamental work \cite{LS} evenso Lusternik and Schnirelmann
worked with the minimal number of {\bf closed} sets. It was \cite{Fox} who showed it for open 
covers. The theorem was extended by Palais to functions $f$ on Banach manifolds satisfying a Palais-Smale
condition and then to continuous functions on some metric spaces as well as flows
$\phi_s$ on topological spaces $X$ satisfying $f(\phi_t(x))< f(\phi_s(x))$ for some continuous
function $f$, for $t>s$ and $x$ which are not fixed points. 
The number of rest points of $\phi_t$ is then at least $\cat(X)$. Since the critical points of
$f$ are rest points of the gradient flow $\phi_s$, this implies the estimate on the number of
critical points. This is the starting point of \cite{RudyakSchlenk} to estimate the number of
fixed points for discrete maps, a framework which could lead to future fixed point applications
for graph endomorphisms which are gradient like. For a recent fixed point theorem, see 
\cite{brouwergraph}. \\

We can now verify that the number of contractible elements in the cover
does not increase. In other words:

\begin{coro}
$\gcat(G)-\tcat(G)$ does not increase when expanding $G$ along simple homotopy steps. 
\end{coro}

\begin{proof}
By lemma~(\ref{pushing}), we can find a category cover of $G$ for which one
$H \cap U_k$ is contractible and the other $H \cap U_j$ are empty for all $j \neq k$. \\
If we make now a pyramid extension over $H$ with the point $x$ over $U_k$, we get a in $G$
contractible set $V_k$ which is contractible if $U_k$ was contractible.
The sets $\{U_j \; \}_{j \neq k}$ together with $V_k$ cover $G'$ and
each of the sets is contractible in $G'$. 
\end{proof}

{\bf Remarks.}  \\
{\bf 1)} If $\tcat(G)=1$, then $\crit(G)=1$ because contractible graphs admit by definition
a function $f$ which has only one critical point. The figure $8$ graph shows that
$\tcat(G)=2$ can be compatible with $\crit(G)=3$. The figure $8$ graph also satisfies 
$\cat(G)=2$ and $\cri(G)=3$. \\
{\bf 2)} At critical points, the topological category can both decrease and increase so that $\crit(G)$ 
can be larger than $\tcat(G)$ in general. The figure 8 is the
smallest example with $\tcat(G)<\crit(G)$. Larger chains show that 
$\crit(G)-\tcat(G)$ as well as $\cri(G)-\cat(G)$ can be arbitrary large. To see that the figure $8$ 
graph has $\crit(G)=3$, note that $\chi(G)=-1$ and that at the minimum of any function $i_f(x)=1$ and at a maximum
$i_f(x)=-1$ or $-3$. In order to add up by Poincar\'e Hopf, there must be an other critical point 
besides the minimum and maximum. Since this argument works for any homotopic graph, we have $\cri(G)=3$ and 
$\cat(G)=2$. \\
{\bf 3)} As in the continuum, the geometric category can be larger than the topological category as the example
of Fox has shown. The dunce hat shows that $\Cat(G)-\tcat(G)$ can decrease when expanding $G$.  \\

\begin{figure}
\scalebox{0.20}{\includegraphics{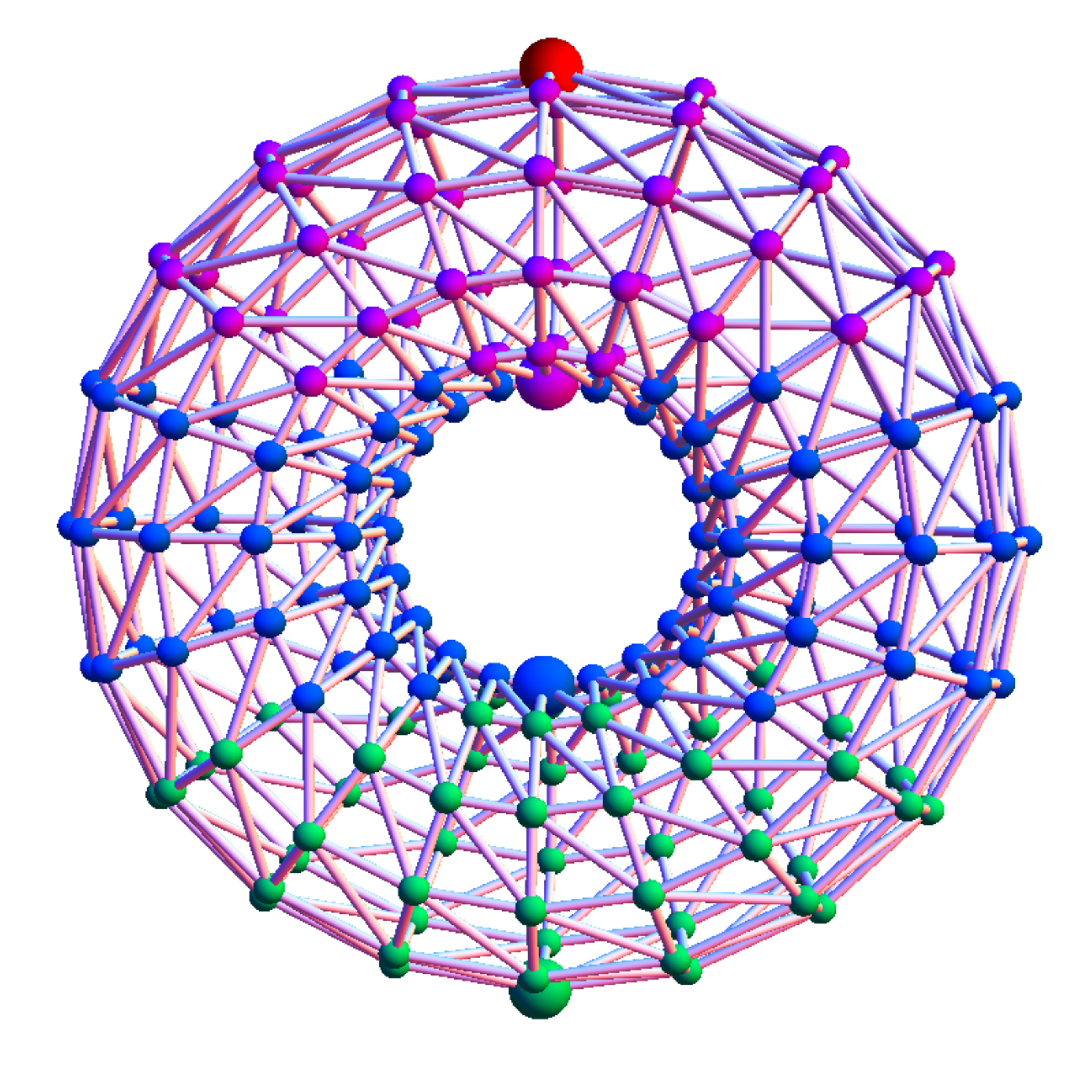}}
\scalebox{0.20}{\includegraphics{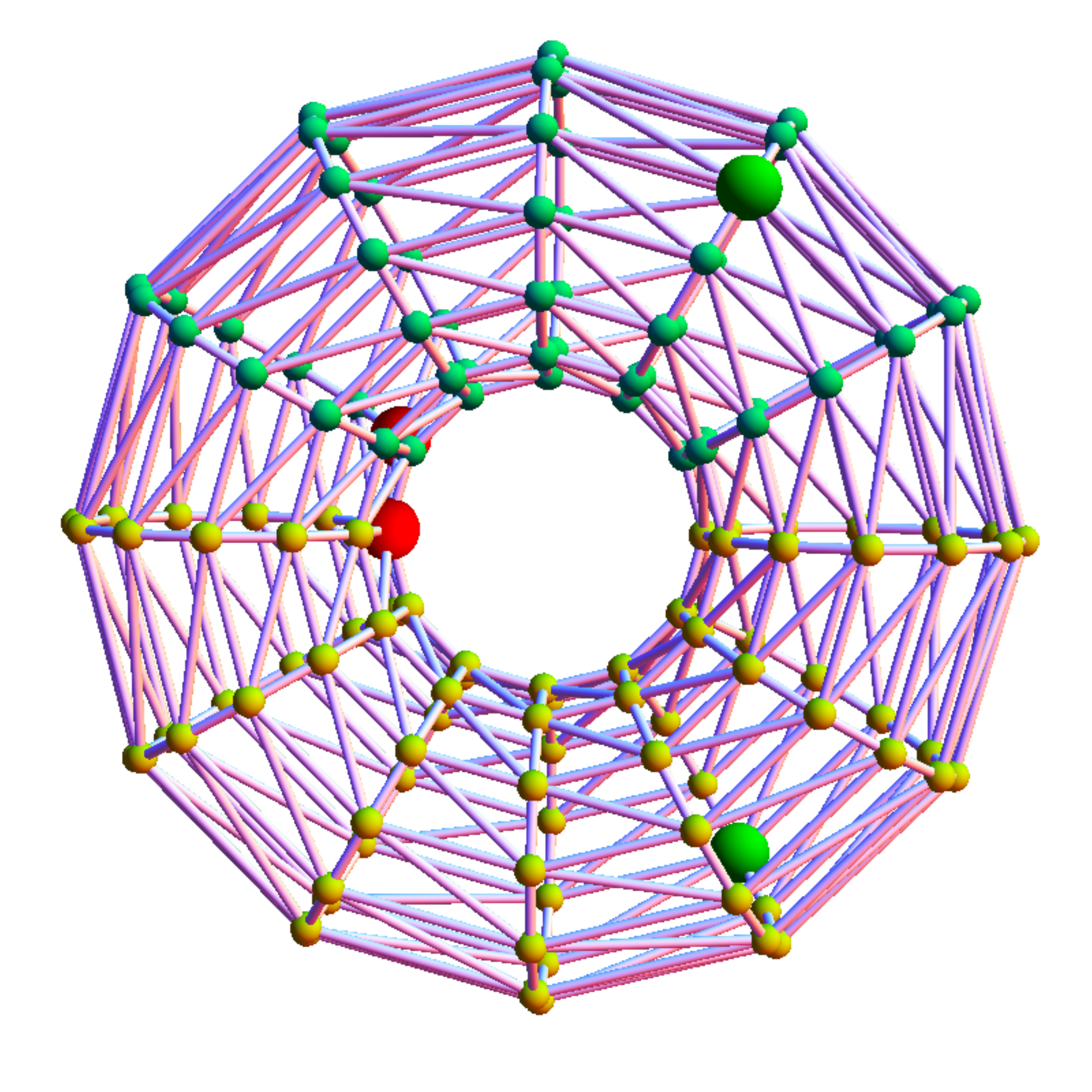}}
\caption{
A toral graph $G$. To the left we see a function $f$ corresponding to the height 
function from an embedding in space. It is a Morse function with 4 critical points,
two of index $1$ and two of index $-1$. 
The second function has $4$ critical points two of which are close together. If 
the two vertices are fused, we get a critical point with 
index $2$ and a function on the torus which has $3$ critical points. The function is
a discretization of $\sin(x) + \sin(y) + \sin(x+y)$ which is also no Morse function in 
the continuum. From $3=\ccup(G) \leq \cat(G) \leq \crit(G)=3$ follows that the category of the 
discrete two dimensional torus is $3$. Its easy to give a category cover. 
If $G=\{1,2,\dots,16 \;\}$ is a hexagonal lattice arranged on a square, then 
$H_1=\{2,3,4,6,7,8,10,11,12 \;\}$, $H_2=\{5,6,8,9,10,11,13,14,15 \;\}$, 
$H_3=\{11,12,9,15,16,13,3,4,1 \; \}$ is a category cover which also verifies $\Cat(G)=3$. 
\label{example2}
}
\end{figure}

{\bf Examples.}  \\
{\bf 1)} Attach a ring at a point x of a sphere. Take the point $x$ away and we need two sets to 
cover. With $x$ we need $3$ sets to cover. The Euler characteristic without the point is $2$ because
there are two components. With the point, it is still $1$ because $b_0=1,b_1=1,b_2=1$. \\
{\bf 2)} Let $x$ be a point of a circular graph. Without the point the Lusternik-Schnirelmann category 
is 1, after it is 2. The Euler characteristic has changed from 1 to 0.\\
{\bf 3)} Let $x$ be a point on the sphere. Without this point, 
the category is $1$, the complete graph has category $2$.
The Euler characteristic changes from $1$ to $2$. \\
{\bf 4)} Assume a figure $8$ graph has the vertex $x$ in the center. Removing 
$x$ does not change the category. It is $2$ both before and after. 
The Euler characteristic changes from $-1$ to $2$ during the separation. Indeed,
for any injective function $f$ which has $x$ as a maximum, $i_f(x)=-3$. \\

Lets look again at the gradient flow associated to an injective function $f:V \to \R$. 
It consists of graphs $G_f(x_j)$ where $G_f(x_j)$ is obtained from $G_f(x_{j-1})$ by attaching
the vertex $x_j$ via a pyramid extension.

\begin{defn}
Given an injective function $f$ on $G$. Define the {\bf category index} $k_f(x)$ of a vertex $x=x_j$ as
$k_f(x_j) = \cat(G_f(x_j))-\cat(G_f(x_{j-1}))$. 
\end{defn}

\begin{propo}
$\cat(G) = \sum_{x \in V} k_f(x)$. 
\end{propo}

Similarly than Poincar\'e-Hopf $\chi(G) = \sum_{x \in V} i_f(x)$, the left hand side is a homotopy
invariant while the right hand side is made of parts $i_f(x)$ which are not. 
Indeed, already the number of vertices is not 
a homotopy invariant. Both $i_f(x)$ and $k_f(x)$ are of course graph isomorphism invariants.
If $\phi:G \to H$ is a graph isomorphism and $g(x)=f(\phi(x0)$ is the injective function on $H$ 
which corresponds to the function $f$ on $G$, then $i_g(\phi(x)) = i_f(x)$ and $k_g(\phi(x)) = k_f(x)$. \\

We have seen that $k_f(x) \neq 0$ implies that $x$ is a critical point.
A critical point does not need to have positive category index as a deformation of a letter $A$ 
graph $G_A$ to figure 8 graph $G_8$ shows. Both before and after the deformation 
the category is $\cat(G_A)=\cat(G_8)=2$. But $\crit(G_A)=2$ and $\crit(G_8)=3$.

\section{Cup length}

The cup length of a graph is an other homotopy invariant. 
It is defined cohomologically as in the continuum. For any finite simple graph $G=(V,E)$,
the vector space $\Omega_k$ of all skew symmetric functions on the set of $k$-dimensional simplices
$\G_k$ plays the role of {\bf differential forms}. The set $\Omega_0$ is the set of functions on 
vertices $V=\G_0$, the set $\Omega_1$ is the set of functions on edges $E=\G_1$. 
It has dimension $v=\sum_k v_k$ where
$v_k$ is the cardinality of the set $\G_k$ of all $k+1$ simplices $K_{k+1}$ in $G$. 
We will call elements in $\Omega_p$ simply {\bf $p$-forms}. In order to define the cup length, 
we need an exterior product which works for general finite simple graphs. 

\begin{defn}
The {\bf pre exterior product} $f \wedge' g$ of a $p$ form $f$ and a $q$ form $g$ is defined as
$f \wedge' g(x_0,\vec{x},\vec{y})$
  $= \sum_{(\sigma,\tau)} (-1)^{(\sigma,\tau)} f(x_0,\sigma(\vec{x})) g(x_0,\tau(\vec{y}))$, 
where $(\sigma,\tau)$ runs over all permutations of $\{1,\dots ,p,p+1,\dots ,p+q \; \}$ where 
$\sigma$ and $\tau$ are both ordered and $(-1)^{(\sigma,\tau)}$ is the sign of the permutation.
The {\bf wedge product} is $f \wedge g(\vec{z}) = 1/(p+q+1) \sum_{\sigma} (-1)^{\sigma} f \wedge g(\sigma(\vec{z}))$,
where $\sigma$ runs over all cyclic rotations of $\vec{z}=\{x_0,x_1,\dots, ,x_{p+q})$. 
\end{defn}

{\bf Remarks.} \\
{\bf 1)}  The pre exterior product is not yet a differential form since a vertex of the simplex is
distinguished. We have to average over the value we get from the other vertices of the simplex too. 
The averaging produces a $p+q$ form, a function on $\G_{p+q+1}$, the set of $p+q+1$ simplices. \\
{\bf 2)} The pre exterior product looks formally close to the definition of the exterior product in the continuum
so that it is good to keep this separate. Think of it as the exterior product centered at $x_0$. 
The averaging process, where each vertex of the simplex becomes the center renders it a function on the 
simplex. \\
{\bf 3)} In small dimensions, especially for computational methods and closer to geometric situations, 
it is possible to produce a discrete differential form calculus for which one has a hodge dual \cite{DKT}.
We don't have a hodge dual for general simple graphs and it looks futile to try in situations where we have
no Poincar\'e duality. \\

{\bf Examples.} \\
{\bf 1)} The exterior product of two one forms $f,g$ is the ``cross product'' of $f$ and $g$. 
It is a function on all triangles $(x_0,x_1,x_2)$ in $G$. We have $k(x_0,x_1,x_2) = f \wedge' g(x_0,x_1,x_2) = 
f(x_0,x_1) g(x_0,x_2) - f(x_0,x_2) g(x_0,x_1)$ and
\begin{eqnarray*}
 3 f \wedge g(x_0,x_1,x_2) &=& f(x_0,x_1) g(x_0,x_2) - f(x_0,x_2) g(x_0,x_1) \\
                           &+& f(x_1,x_2) g(x_1,x_0) - f(x_1,x_0) g(x_1,x_2) \\
                           &+& f(x_2,x_0) g(x_2,x_1) - f(x_2,x_1) g(x_2,x_0)  \; . 
\end{eqnarray*}
{\bf 2)} The exterior product of a $2$-form $k$ and a $1$-form $h$ gives the pre exterior product
$k\wedge'h(x_0,x_1,x_2,x_3)$=$k(x_0,x_1,x_2)h(x_0,x_3)$-$k(x_0,x_1,x_3)h(x_0,x_2)$+$k(x_0,x_2,x_3)h(x_0,x_1)$. Now
average this over all possible rotations of $x_0,x_1,x_2,x_3$ with the sign factor of the rotation.
This product can only be nonzero at vertices $x_0$ which are part of a $K_4$ clique.  \\
{\bf 3)} Given three $1$-forms $f,g,h$, we can use the previous two examples to get the 
``triple scalar product'' $f \wedge g \wedge h$ which is deduced from
\begin{eqnarray*}
  f\wedge' g\wedge' h(x_0,x_1,x_2) &=& 
      k(x_0,x_1,x_2)h(x_0,x_3)-k(x_0,x_1,x_3)h(x_0,x_2) \\ && \hspace{3cm} +k(x_0,x_2,x_3)h(x_0,x_1)\\
      &=&[f(x_0,x_1)g(x_0,x_2)-f(x_0,x_2)g(x_0,x_1)]h(x_0,x_3)  \\
      &-&[f(x_0,x_1)g(x_0,x_3)-f(x_0,x_3)g(x_0,x_1)]h(x_0,x_2)  \\
      &+&[f(x_0,x_2)g(x_0,x_3)-f(x_0,x_2)g(x_0,x_3)]h(x_0,x_1)  \\
      &=&\det  \left[ \begin{array}{ccc}
          f(x_0,x_1)&f(x_0,x_2)&f(x_0,x_3) \\
          g(x_0,x_1)&g(x_0,x_2)&g(x_0,x_3) \\
          h(x_0,x_1)&h(x_0,x_2)&h(x_0,x_3) 
                     \end{array} \right]  
\end{eqnarray*} 
and 
\begin{eqnarray*}
3f \wedge g \wedge h(x_0,x_1,x_2) &=& 
\det  \left[ \begin{array}{ccc}
          f(x_0,x_1)&f(x_0,x_2)&f(x_0,x_3) \\
          g(x_0,x_1)&g(x_0,x_2)&g(x_0,x_3) \\
          h(x_0,x_1)&h(x_0,x_2)&h(x_0,x_3) 
                     \end{array} \right] \\
&+& \det  \left[ \begin{array}{ccc}
          f(x_1,x_2)&f(x_1,x_3)&f(x_1,x_0) \\
          g(x_1,x_2)&g(x_1,x_3)&g(x_1,x_0) \\
          h(x_1,x_2)&h(x_1,x_3)&h(x_1,x_0) 
                     \end{array} \right] \\
&+& \det  \left[ \begin{array}{ccc}
          f(x_2,x_3)&f(x_2,x_0)&f(x_2,x_1) \\
          g(x_2,x_3)&g(x_2,x_0)&g(x_2,x_1) \\
          h(x_2,x_3)&h(x_2,x_0)&h(x_2,x_1) 
                     \end{array} \right] \; .
\end{eqnarray*}
{\bf 4)} The pre exterior product between the $1$-forms $f_1,\dots ,f_n$ is $f_1 \wedge' \cdots \wedge' f_n(x_0) = {\rm det}(A)(x_0)$,
where $A_{jk} = f_j(x_0,x_k)$. It can only be nonzero at vertices $x_0$ which are contained in a clique
$K_{n+1}$. The exterior product is an average of $n+1$ such determinants. \\
{\bf 5)} On a graph without triangles like the cube graph or the dodecahedron or for trees,
the exterior product is trivial $f \wedge g=0$ for all forms $f,g$ because there are no triangles
in those graphs.  \\
{\bf 6} Lets take the graph $G=K_3$. We have for example $(3,4,5) \wedge (2,1,4)  = 4/3$.
If we chose a basis $i=(1,0,0),j=(0,1,0),k=(0,0,1)$ and denote by $t$ the 2-form at the triangle.
and $T=t/3$, we can write down the multiplication table
\begin{center}
\begin{tabular}{l|llll}
    & i & j & k & t \\ \hline
 i  & 0 & T & -T& 0 \\
 j  &-T & 0 & T & 0 \\
 k  & T &-T & 0 & 0 \\
 t  & 0 & 0 & 0 & 0 \\
\end{tabular}
\end{center}

The exterior product defines a differential graded algebra on $V = \Omega_0 \oplus \Omega_1 \oplus \dots \oplus \Omega_n$,
where $K_{n+1}$ is the largest simplex which appears in the graph $G$. Its dimension is the sum 
$v=\sum_{j=0}^n v_j$ and the Euler characteristic of the graph is the super sum 
$\chi(G) = \sum_{j=0}^n (-1)^j v_j$. The product is associative and satisfies the same super anti-commutation 
relation $f \wedge g = (-1)^{p q} g \wedge f$ and the Leibniz rule $d(f \wedge g) = df \wedge g) + (-1)^p f \wedge dg$ 
as in the continuum. To see this, note that the pre-wedge product
has this property which is inherited by the sum. We only have to get used to the
fact that the ``tangent spaces'' at different vertices can vary from vertex to vertex and that we have
to average in order to have a meaningful product. \\

One forms are associated with functions on directed edges $e=(a,b)$ in such a way that
$f(a,b)=-f(b,a)$. As in the continuum, products of one forms can be used to generate $\Omega_p(G)$. 
Given an oriented edge $e=(x,y)$, we have a one form $f_{x,y}$ for which $f_{x,y}(x,y)=1, f_{x,y}(y,x)=-1$
and $f_{x,y}(p,q)=0$ for any other edge $(p,q)$ different from $(x,y)$. These one forms span $\Omega_1$
and the product $f_1 \wedge \cdots \wedge f_p$ of such one forms (attached to 
the same vertex to be nonzero) spans the vector space $\Omega_p$. These products play the analogue
of compactly supported differential forms $\rho(x) dx_{k_1} \wedge \cdots \wedge dx_{k_p}$ in the continuum. \\

{\bf Remarks.} \\
{\bf 1)} In the continuum, for manifolds, tangent spaces at different points
are isomorphic, this is not the case for graphs which have non-constant degree. Since ``tangent spaces" overlap
in the discrete, we have to symmetrize the pre wedge product. The graphs are arbitrary finite simple graphs $G=(V,E)$. 
While in the continuum, for $n$ dimensional manifolds $M$, the vector space $\Omega_k(x)$ has dimension 
$\left( \begin{array}{c} n \\ k \end{array} \right)$ at every point,
the dimensions $v_k(x)$ of $\Omega_k(x)$ can be different numbers and the sum $v_k=\sum_{x \in V} v_k(x)$
can be pretty arbitrary. \\
{\bf 2)} Many wedge products are zero. 
The {\bf cross product} $f \wedge g$ of two one forms $f,g$ for example can only
be nonzero if there is a triangle at that vertex.
For planar graphs, there are no hyper tetrahedral subgraphs $K_5$, 
so that the product of two 3-forms is always zero for planar graphs.

\begin{defn}
The {\bf exterior derivative} $d: \Omega_k \to \Omega_{k+1}$ is defined as 
$df(x_0,\dots, x_k) = \sum_{j=0}^k (-1)^j f(x_0,\dots,\hat{x}_j,\dots,x_k)$. Since
$d \circ d(f)=0$, the vector space $\ran(d_{k-1})$ of coboundaries is contained in 
the vector space $\ker(d_k)$ of cocyles and defines the vector space 
$H^k(G)=\ker(d_k)/\ran(d_{k-1})$ called $k$'th cohomology group of $G$. Its dimension $b_k$
is called the k'th {\bf Betti number}. The vector $(b_0,b_1, \dots, )$ is called the 
{\bf Betti vector} of $G$, the polynomial $\sum_{k=0}^{\infty} b_k t^k$ the {\bf Poincar\'e polynomial}.
\end{defn}

It follows from the Rank-nullety theorem in linear algebra and cancellations
by summing up that the cohomological Euler characteristic
$\sum_{k=0}^{\infty} (-1)^k b_k$ agrees with the combinatorial Euler 
characteristic $\sum_{k=0}^{\infty} (-1)^k v_k$ for any finite simple graph.
The associative product $\wedge$ which defines the {\bf exterior algebra} $(\Omega,\wedge)$ 
on a finite simple graph induces a {\bf cup product} $H^k(G) \times H^l(G) \to H^{k+l}$ on the
equivalence classes by $[f] \wedge [g] = [f \wedge g]$. The fact that the wedge product of 
two coboundaries is a coboundary follows from the Leibniz rule $d(f \wedge g) = df \wedge g 
+ (-1)^p f \wedge dg$ for a $p$ and $q$ form. The cup product is therefore well defined as
in the continuum. 

\begin{defn}
The minimal number $m$
of $k$-forms $f_j$ with $k \geq 1$ in this algebra with the property that
$f_1 \wedge f_2 \cdots \wedge f_k$ is always zero in $H^m(G)$ is called the {\bf cup length} of
the graph $G$. It is denoted by $\ccup(G)$. 
\end{defn}

{\bf Examples.} \\
{\bf 1)} The cup length of $K_n$ is $1$ because only $H^0(G)$ is nonzero so that any $k$ form $f$
with $k \geq 1$ is already zero by itself in $H^k$. Similarly the cup length of any contractible
graph is $1$ and agrees with $\cat(G)$. More generally, the cup length of any graph homotopic to a one
point graph is $1$. \\
{\bf 2)} The cup length is $1$ for any contractible graph,
as homotopy deformations do not change cohomology. \\
{\bf 3)} The cup length is $2$ for any discrete sphere like an octahedron or icosahedron
 or the 16-cell, the 3 dimensional cross polytop which has $8$ vertices, $12$ edges,
$8$ faces and $16$ tetrahedral cells. The reason is that the Betti vector is $(1,0,\dots,0,1)$
and there is always a volume form $f$ which is nonzero. Since there is no other $k$ form 
$k\geq 1$ available, any product of two forms is zero. \\
{\bf 4)} The cup length is equal to $d+1$ for any discrete torus $G=T^d$. 
The reason is that $H^1(G)$ is $d$ dimensional a basis $f_1,\dots,f_d$ 
of which generates all cohomology classes. $H^p(G)$. Since 
$f_1 \wedge \cdots \wedge f_d$ is a basis for the one dimensional space of volume forms
and so nonzero, $\ccup(G) \geq d+1$. But since any product of $d+1$ or more $k$-forms
with $k \geq 1$ is zero, $\ccup(G) \leq d+1$.

\begin{propo}
The cup length is a homotopy invariant. 
\end{propo}
\begin{proof}
Cohomology does not change under elementary homotopy deformations.
If $G_1$ and $G_2$ are
homotopic and $f_1 \wedge \cdots \wedge f_n=0$, then the corresponding 
$g_1 \wedge \cdots \wedge g_n=0$ in $G_2$. Therefore $\ccup(G_2) \leq \ccup(G_1)$. 
Reversing this shows $\ccup(G_2)=\ccup(G_1)$. 
\end{proof} 

\begin{thm}[Cup length estimate]
$\ccup(G) \leq \cat(G)$.
\end{thm}

The proof is the same as the continuum (see e.g. \cite{Bott82} page 341 or 
the introduction to \cite{CLOT}).

\begin{proof}
Assume $\cat(G)=n$. 
Let $\{U_k \; \}_{k=1}^n$ be a Lusternik-Schnirelmann cover of $G$.
Given a collection of $k_j \geq 1$-forms $f_j$ with $f_1 \wedge f_2 \dots \wedge \dots f_n \neq 0$.
Using coboundaries  we can achieve that for any simplex $y_k \in U_k$, we can gauge $f$ so that   
$f(y_k)=0$. Because $U_k$ are contractible in $G$, we can render $f$ zero in $U_k$. This shows that 
we can chose $f_k$ in the relative cohomology groups $H^k(G,U_k)$ meaning that we can 
find representatives $k_j$ forms $f_j$ which are zero on each $p_{k_j}$ simplices 
in the in $G$ contractible sets $U_k$. But now, taking these  
representatives, we see $f_1 \wedge \cdots \wedge f_n = 0$. This shows $\ccup(G) \leq n$. 
\end{proof}

{\bf Examples.}\\
{\bf 1)} We have seen that for contractible graphs $\ccup(G) = \cat(G) = \crit(G) = 1$. \\
{\bf 2)} For spheres, for which the Betti vector is $(1,0, \dots, 0,1)$, we have $\ccup(G)=\cat(G)=\crit(G)=2$. \\
{\bf 3)} For a triangularization of $T^n$, the Betti vector is $(1,n,n(n-1)/2,\dots, n,1)$ 
and $\ccup(G) \geq n+1$ because we can find $n$ one-forms whose product is not zero in 
$H^n(G)$. Because $\crit(G)=n+1$ we have $\ccup(G)=\cat(G)=\crit(G)=n+1$. 
In the continuum, the standard example of a function on the $2$-torus with 3 critical points is 
$f(x,y) = \sin(x)+\sin(y) + \sin(x+y)$ an example which discretizes. For this function,
the index of one of the critical point (monkey saddle type) is $-2$ and the other two 
critical points (max and min) have index $1$.  \\
{\bf 4)} Lets see how this look in detail for a cyclic graph $C_4$ which has $4$ vertices
and $4$ edges. There is a nonzero $1$ form $f$ which is equal to $1$ for any edge $(i,i+1)$, satisfies
$df=0$ but which is not a gradient $f=dg$. 
Therefore $\ccup(G) \geq 2$. To see that $\ccup(G) \leq 2$, we have to show that the wedge product of 
any two $1$ forms $f,g$ is zero. Take the category cover $U_1 = \{1,2,3 \; \}$ and $U_2 = \{3,4,1 \; \}$
Take coboundaries = gradients $dh,dk$ with $h=(0,f(1,2),f(2,3)-f(1,2),0)$ and
$k=(0,0,g(3,4),g(4,1)-g(3,4))$. Then the representatives $F=f-dh=(0,0,*,*)$ and $G=g-dk = (*,*,0,0)$ of 
the cohomology classes have the property that $F \wedge G = 0$. This 
means that the cup product of $f$ and $g$ is zero. \\

\begin{coro}
For any injective function $f$, the cup length of a graph is bounded above 
by the number of critical points of $f$.
\end{coro}
\begin{proof}
Combining the two theorems leads to the estimate 
$$  \ccup(G) \leq \cat(G) \leq \crit(G) $$
between two homotopy invariants and one simple homotopy invariant of a graph. 
\end{proof}

\section{The dunce hat}

The dunce-hat is a two dimensional space of Euler characteristic $1$. 
It has been proven important in the continuum with relations up to the 
Poincar\'e conjecture and it is pivotal also here to understand the 
boundaries of the Lusternik-Schnirelmann theorem $\ccup(G) \leq \tcat(G) \leq \crit(G)$.
It is a cone where the boundary rim is glued to a radius. Like the Bing house of two rooms
(for which a graph theoretical implementation was given in \cite{CYY}), it is an example
of a space which is homotopic to a point but not collapsible in itself to a point. 
It is not a manifold: some unit spheres are figure 8 graphs of Euler characteristic
$-1$. The space has been introduced in \cite{Zeeman64}. A graph theoretical implementation  $G$
with $v_0=17$ vertices, $v_1=52$ edges and $v_2=36$ triangles 
was given in \cite{BouletFieuxJouve} and shown in Figure~\ref{duncehat}.
The unit spheres are all one dimensional. Some of them are 
circles $C_4,C_6,C_{12}$ but there are others which are homotopic to a 
figure 8 graph. Every vertex has dimension $2$ in the graph 
theoretical sense since each unit sphere is one dimensional in the graph theoretical sense. 
Since some spheres are not circles, it is not a polyhedron in the graph theoretical sense. 

\begin{figure}
\scalebox{0.20}{\includegraphics{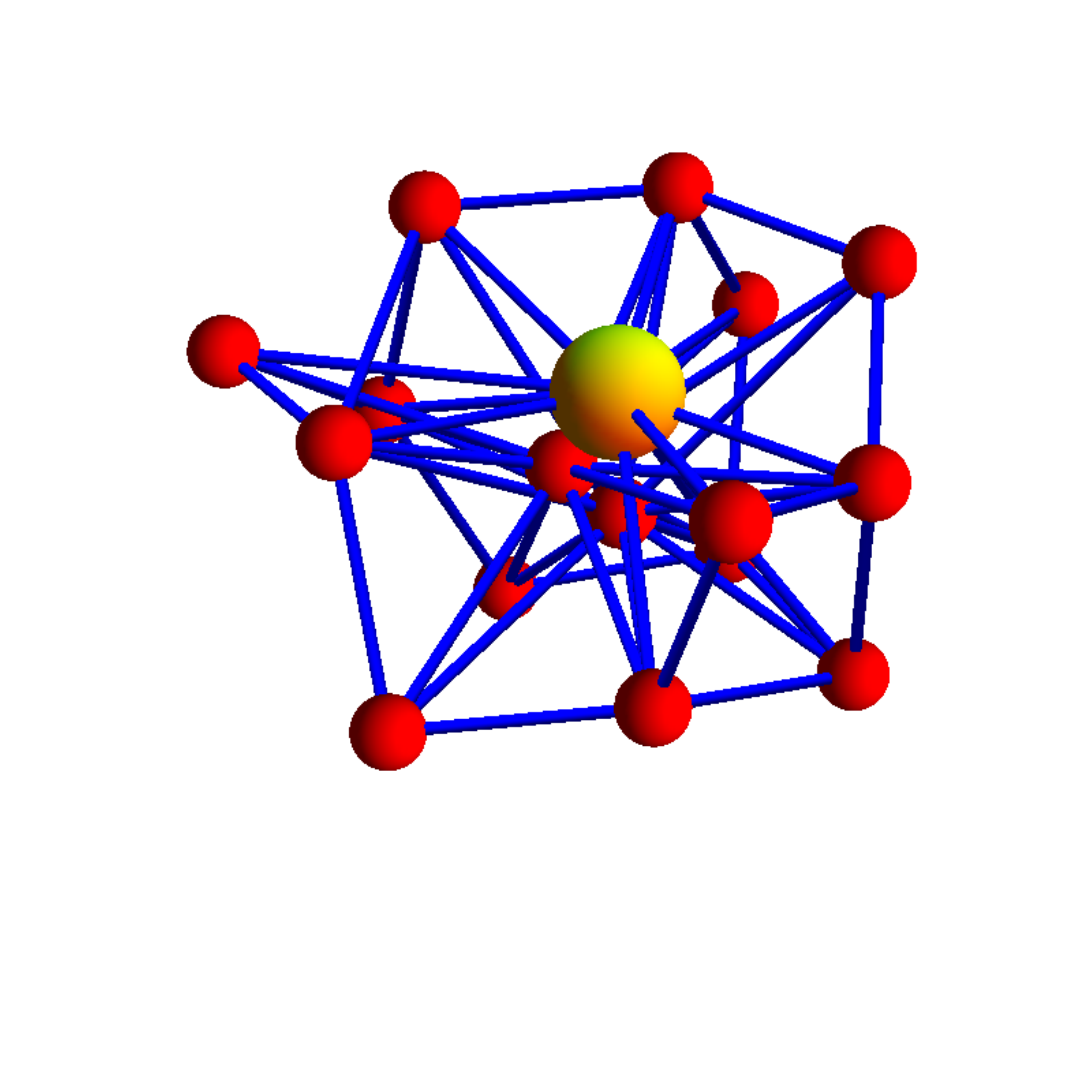}}
\scalebox{0.20}{\includegraphics{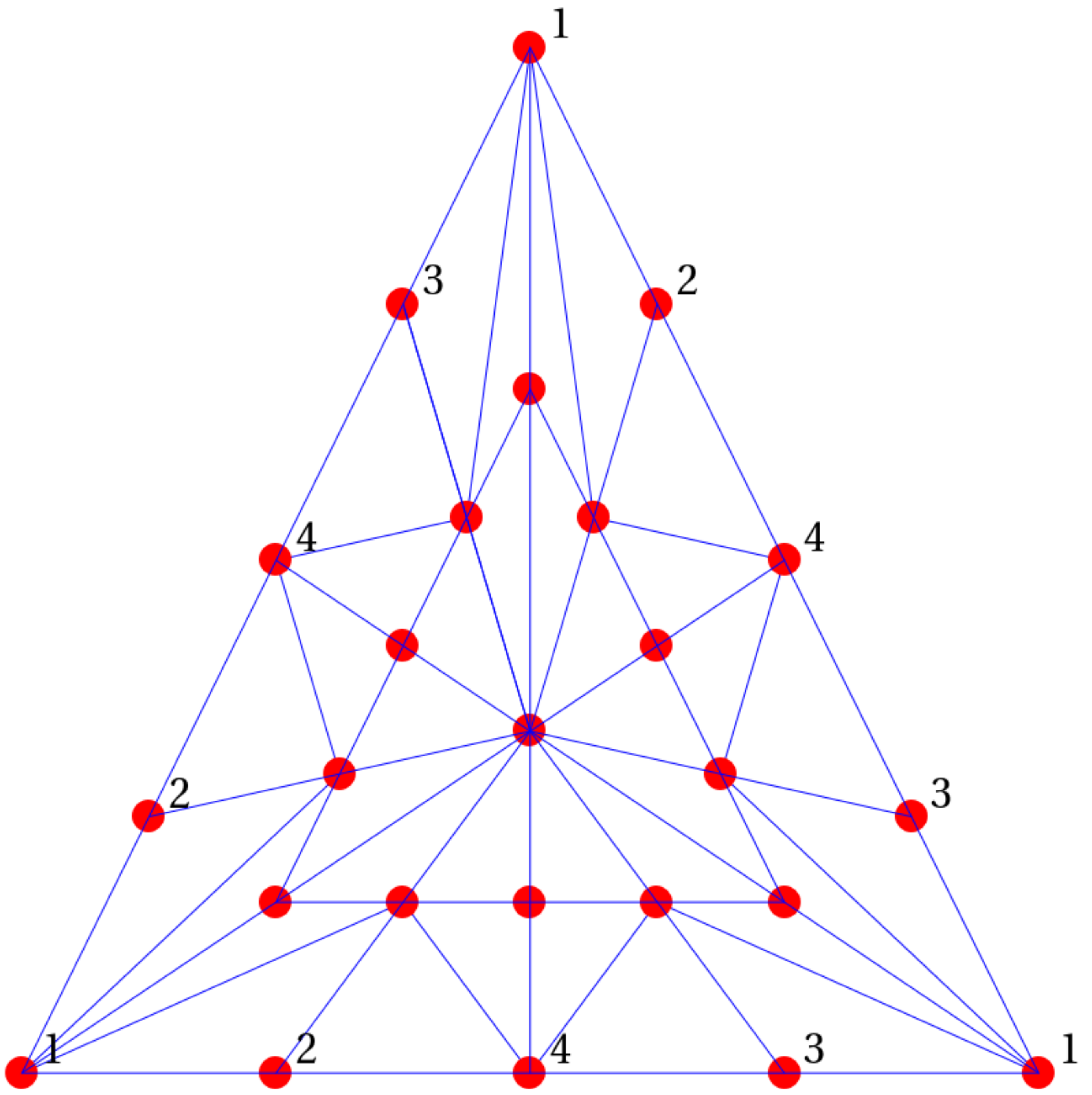}}
\caption{
A graph version of the dunce-hat \cite{Zeeman64}.  This implementation with 17 vertices
had been used in \cite{BouletFieuxJouve}.
It is an example of a graph which has the cohomology of a point, which has
category $2$ (because it is not contractible) and where the minimal number of critical points is $3$
(because for two dimensional graphs every critical point has nonzero index and the sum of indices is the
Euler characteristic $1$ showing that besides the minimum of index $1$ at least two other critical 
points have to be present). Since the graph is not contractible to a point but 
homotopic to a point, the homotopy needs first to expand the graph. Indeed $G$ is a homotopy
retract of a three dimensional contractible graph.
\label{duncehat}
}
\end{figure}

\begin{propo}
The dunce-hat graph $G$ satisfies
a) $\ccup(G)=1$, b) $\tcat(G)=2$ and c) $\crit(G)=3$
Furthermore, $\cat(G)=1$ and $\cri(G)=1$. 
\end{propo}
\begin{proof}
a) The graph is connected and simply connected so that $H^0(G)=R$ and $H^1(G)=0$. The Betti
vector is $(1,0,0,\dots)$ so that the space is star shaped = ``a homology point''. Because
for every $k \geq 1$ form $f$ we have $f=0$, the cup length is $1$. \\
b) It is possible to cover the graph with two contractible subgraphs generated
by the vertices $V_1=V \setminus \{1,3\}$ and $V_2=\{1,3\}$. Therefore $\ccup(G) \leq 2$. 
If $G$ were contractible, there would be an injective function on the vertex set with
only one critical point. Looking through all possible cases shows that this is not the case.
Therefore, $\ccup(G) \geq 2$.  \\
c) There are concrete functions which have three critical points. 
The minimum has index $1$. Each critical point must have nonzero index 
because the only subgraph of any of the unit sphere $S(x)$ only has index
$1$ if it is contractible meaning that the vertex $x$ is then not a critical
point. Because the sum of the indices is $1$ by Poincar\'e-Hopf, there must be at least $3$
critical points. 
d) and e) follow from homotopy invariance and the fact that $G$ is homotopic to a one point graph.
\end{proof}

{\bf Remarks.} \\
{\bf 1)} The dunce hat graph is an example showing that $\crit(G)$ is not necessarily 
invariant under homotopies. While expanding or shrinking a graph, the number of
critical points of the corresponding function does not change, but there might be a
different function available which has less critical points. \\
{\bf 2)} The still open Zeeman conjecture which implies the (now proven) Poincar\'e conjecture in the
continuum states that the product of any space homotopic to a point with an interval is collapsible in the Whitehead
sense after some barycentric subdivisions. In a graph theoretical sense, the Zeeman conjecture would be:
for any graph $G$ which is homotopic to the identity, any new graph $H$ obtained by triangulating first the direct
product graph $G \times K_2$ and then refine the triangularization with pyramid extensions over
simplices, is contractible. 

\section{Curvatures}

Like Euler characteristic or cohomology, category theory is important because it
provides an other homotopy invariant for graphs $G=(V,E)$. For each of the numerical invariants
$\chi(G),b_k(G),\cat(G)$ and any injective function $f$ on the vertex set $V$ we
can build up the graph given the ordering defined by $f$ and keep track on how the 
Euler characteristic, Betti numbers or category change. We get so 
an index $i_f(x),j_{k,f}(x),k_f(x)$ at every point and Poincar\'e-Hopf type formulas
$$ \chi(G)    = \sum_{x \in V} i_f(x) \; , 
   b_k(G) = \sum_{x \in V} j_{f,k}(x) \; ,
   \cat(G) = \sum_{x \in V} k_f(x)    \; .   $$
The left hand side is always a homotopy invariant, while the components to the
right can also depend on a function $f$ and the actual implementation of the graph $G$.
Since category can increase maximally by $1$ when adding a critical points, we have
$k_f(x) \leq 1$. But $k_f(x)$ is not bounded below as a pyramid extension over the whole graph shows,
where the category, however large, collapses to $1$. \\

By averaging the index over all injective functions where the values $f(x)$ are independent, identically distributed
random variables with uniform distribution in $[-1,1]$, we get curvatures and Gauss-Bonnet
type theorems
$$  \chi(G)  = \sum_{x \in V} K(x)    \; , 
    b_k(G)   = \sum_{x \in V} B_k(x)  \; , 
    \cat(G)  = \sum_{x \in V} C(x)    \; . $$
The Euler curvature $K(x)$ has been defined in the continuum and the identity is called the 
Gauss-Bonnet-Chern theorem. The other curvatures
seem not have been appeared, also not in the continuum.
But {\bf Betti curvatures} $B_k$ and {\bf category curvature} $C$ are not local, unlike
the Euler curvature $K(x)$.  \\

{\bf Examples} \\
{\bf 1)} Look at a linear graph $I_n$ with vertex set $V=\{x_1,x_2,\dots ,x_n \; \}$ 
with $n$ vertices and look at a function $f$ on the vertex set $V$. 
At any local minimum we have $k_f(x)=1$ and every local maximum except at 
the boundary we have $k_f(x)=-1$. We see that $C(x)=1/2$ at the boundary points.  \\
{\bf 2)} For a cyclic graph $C_n$, the index $k_f(x)$ is $1$ at a local minimum and $-1$
at a local maximum except at the global maximum, where $k_f(x)=1$ too. We have $C(x)=2/n$.
This is also to be expected due to the fact that $\cat(G)=2$ and $C(x)$ is constant.
This example shows that the index can not be local because the curvature depends on $n$,
while the local neighborhood does not. \\
{\bf 3)} For a complete graph $K_{n+1}$, the index $k_f(x)$ is $1$ at a local minimum and $-1$ at
a local maximum except at the global max. We have $k_f(x)=1/(n+1)$. Again this also has
to be expected from symmetry.  \\
{\bf 4)} By symmetry, the category curvature at a vertex of the octahedron is $1/4$ and the 
category curvature of an icosahedron $1/6$ at every vertex.

\section{Questions}

{\bf A)} The Euler-Poincar\'e formula $\chi(G) = \sum_{j=0}^{\infty} (-1)^j b_j$
shows that a change of the Euler characteristic is necessarily linked to a 
change of cohomology. On the other hand, a cohomology change of $G(x)$ implies that $x$ is a critical point. 
Lets call a graph {\bf star shaped} if it is connected and only $H^0(G)$ is nontrivial.
The dunce-hat $G$ is star shaped and homotopic to the identity but it is not contractible: $\cat(G)=2$.
One can ask about characterizing such discrepancies.
The case of homology spheres in the continuum shows that cohomology does not 
determine spaces up to homotopy. Is there a systematic way to construct graphs for which $\cat(G)-\ccup(G)$
is a given positive number? 
The relation between cohomology, category and critical points might also allow to give 
upper and lower estimates for $h(n)$ the number of homotopy classes of finite simple graphs 
of order $n$. We have mentioned that the computation of $h(n)$ is difficult with
$2^{n(n-1)/2}$ graphs of order $n$. Computing the number of homotopy classes of $\cat(G)=k$,
$\ccup(G)=k$ or $\crit(G)=k$ graphs might help to split up the task.  \\

\begin{figure}
\scalebox{0.13}{\includegraphics{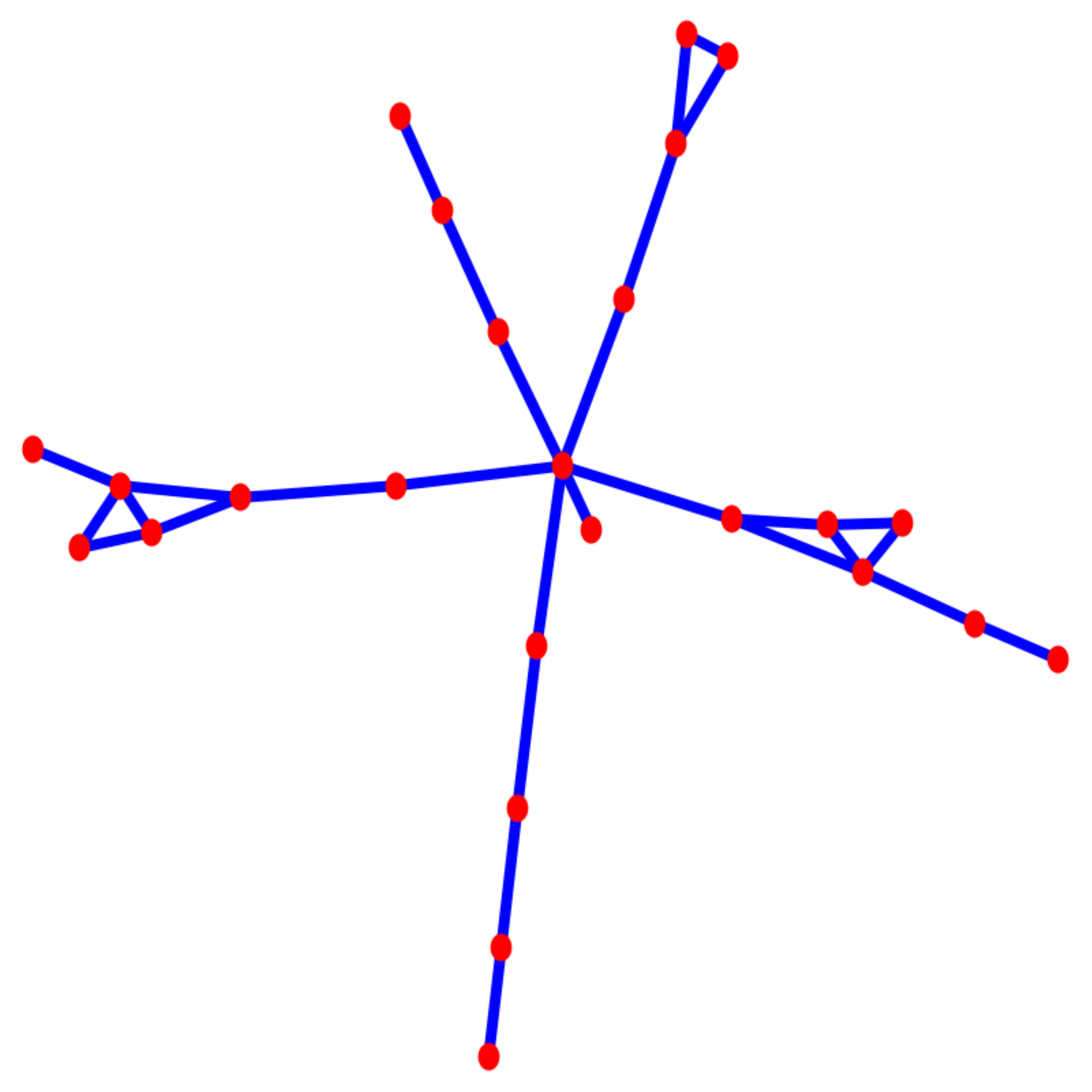}}
\scalebox{0.13}{\includegraphics{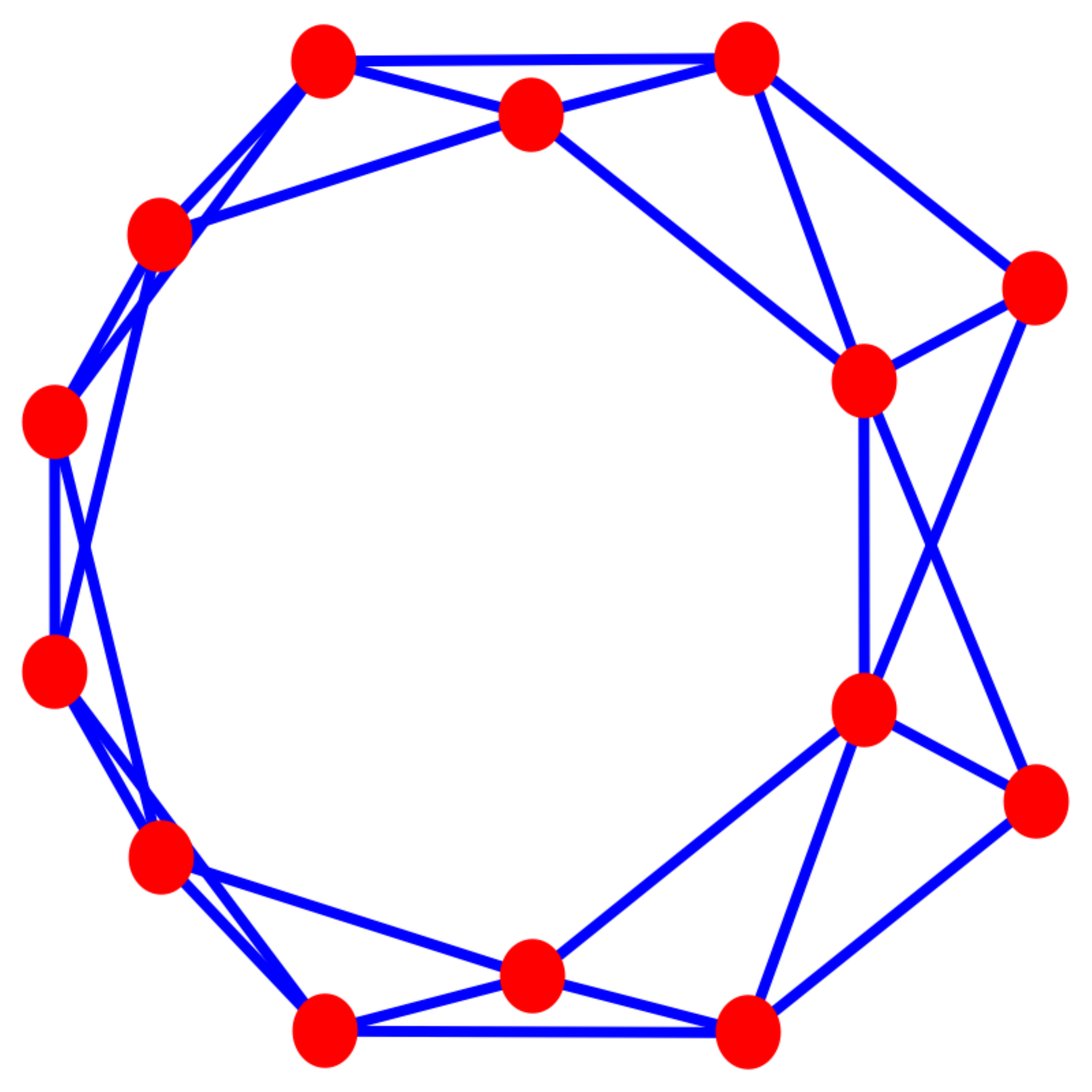}}
\scalebox{0.13}{\includegraphics{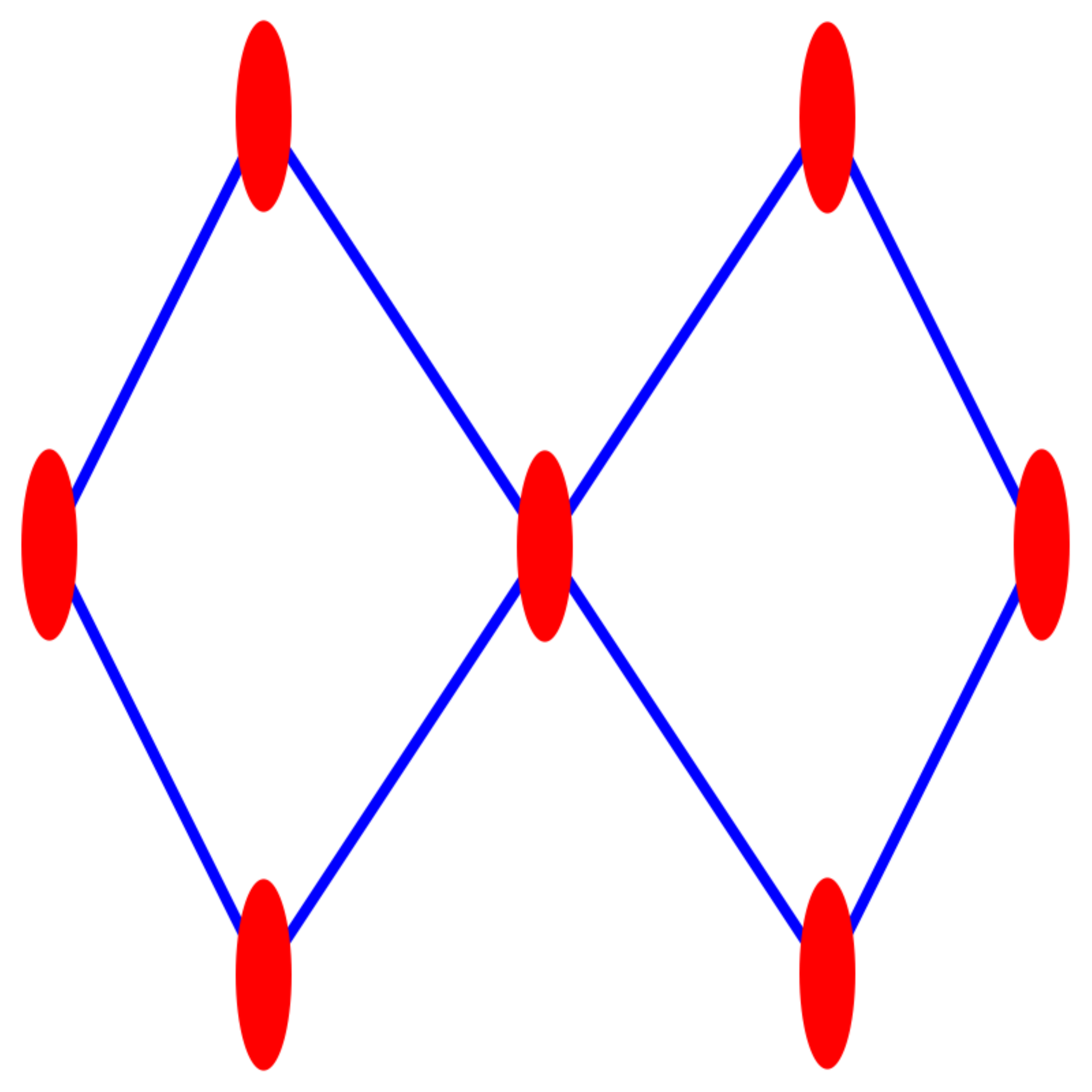}}
\caption{
The left figure shows an example of a cat-1 graph. 
The middle figure shows an example of a cat-2 graph of dimension
1. It is homotopic to a circular graph. The right figure 8 graph 
is a cat-2 graph for which $\crit(G)=3$ is larger. The figure $8$ graph $G$
is also an example of a graph for which the removal of any vertex does
not change the topological category. We have $\ccup(G)=2$ because $H^2(G)=0$ and 
$b_1>0$.  The figure 8 graph is an example where $\cat(G)<3=\crit(G)$.  
It has Betti vector $(1,2)$ and Euler characteristic $-1$. 
\label{smallcat}
}
\end{figure}

{\bf B)} Can one characterize finite simple connected graphs with fixed $\cat(G)=k$
or $\crit(G)=k$? By definition, cat-1 graphs are the contractible graphs which is a subclass of 
graphs homotopic to a one point.  This includes
star graphs, connected trees or complete graphs. 
The cat-2 graphs contain ``spheres'' which can be obtained by gluing together two contractible sets
along a sphere but it also contains examples where two spheres are glued along a point. And there
are examples like the dunce hat of $\tcat(G)=2$ graphs which are homotopic to a point 
satisfying therefore $\cat(G)=1$ but which is not contractible.
The Fox graph is an example of a tcat-2 graph which is not a sphere.
The statement that cat-2 manifolds have a free fundamental group \cite{DKR} might be true also for graphs.
It also might be true that every finitely presented group occurs in a graph of category $3$. 
\cite{DKR} ask a continuum version of the question whether $cat(M-p)=1$ for every closed manifold with $\cat(M)=2$. 
This is not true for graphs as the figure 8 graph shows. Removing any point still keeps category $2$. 
In general, it might be interesting to study homotopy groups and "weakly contractible graphs", graphs for which
all homotopy groups are trivial.  $\crit(G)=2$ graphs are spheres by a discrete Reeb theorem 
and have the Betti numbers $\vec{b} =(b_0,0, \dots, 0,b_n) = (1,0, \dots 0,1)$ and 
Euler characteristic $1+(-1)^n$. The smallest connected cat-2 graph is $C_4$. 
Characterizing the homotopy class of cat-3 or crit-3 graphs could already be not easy. Takens \cite{Takens}
classified three dimensional crit-3 manifolds: they are iterated connected sums $N_1^{\sharp k}$ of $N_1=S^1 \times S^2$ or 
of the form $N_1^{\sharp k} \sharp N_2$ where $N_2$ is the total space of a nonorientable $S^2$ bundle over $S^1$. The
connected sum $M=A \sharp B$ is obtained by taking a disjoint union of $A,B$ by removing the interior
or a closed 3-cell in the interior of $A$ and $B$ and identifying the boundaries by a diffeomorphism. Takens proof
consists of 4 statements, some of which can be discretized. But the connected sum in the discrete obtained by gluing
along a graph automorphism instead of a diffeomorphism could lead to more general graphs.  \\

{\bf C)} Finally, as the computation of category is acknowledged to be hard in the continuum \cite{CLOT},
one can ask about the complexity to compute $\cat(G)$ or $\crit(G)$ in graphs of order $n$. Is this problem in 
$NP$? Indeed we only know of clumsy ways to compute $\tcat(G)$: first deform $G$ to make it as small
as possible, then make a list of maximal contractible sets and see how many are needed to cover $G$. This
is a task which grows exponentially with the length=order+size of the graph. The lower and upper bounds
$\ccup$ and $\crit$ are not difficult to implement, but can be expensive to compute in 
general for a general finite simple graph. 

\section*{Appendix: Notions of homotopy}

Topologists often look at graphs as one dimensional simplicial complexes. Any connected graph $G$
is then homotopic to the wedge sum of $n=\chi(G)-1$ cycle graphs. There is obviously no relation at all
with the Whitehead type notions introduced by Ivashkenko to graph theory which is used here. \\

The following definition was done first in \cite{I93,I94} for graphs. As the definition of contractible, it 
can be done using induction with respect to the {\bf length} $n = |V| + |E|$ which is the sum of the order and
size of the graph. Assume the definition has been done for all graphs of length smaller than $n$,
we can use it to define contractible and homotopy for graphs of length $n$. 

\begin{defn}
An {\bf I-homotopy deformation step} is one of the four following steps: \\
{\bf a)} Deleting a vertex $v$ together with all connections if $S(x)$ is I-contractible. \\
{\bf b)} Adding a vertex with a pyramid construction over an I-contractible graph $H$. \\
{\bf c)} Adding an edge between two vertices $v,w$ for which $S(v) \cap S(w)$ is I-contractible. \\
{\bf d)} Removing an edge between two vertices $v,w$ for which $S(v) \cap S(w)$ is I-contractible. \\ \
\vspace{3mm}
Two graphs are called {\bf I-homotopic} if they can be transformed into each other by I-homotopy steps. 
A graph is {\bf I-contractible}, if it is I-homotopic to a graph with a single vertex. 
\end{defn}

As noticed in \cite{CYY}, the vertex deformation steps a) and allow to do the 
edge deformation steps c) and d):

\begin{propo}[Chen-Yau-Yeh]
Two graphs are I-homotopic if and only if they are homotopic:
they can be transformed into each other by using vertex deformation steps a) and b) alone.
\end{propo}

\begin{proof}
We only have to show that we can realize $c)$ using deformation steps a) and b).
The inverse d) can then be constructed too:
if $c = \prod_{j=0}^{n-1} c_j$ with $c_j \in \{ a,b \; \} $ then 
$d = c^{-1} \prod_{j=1}^n c_{n-j}^{-1}$ can  also be realized using 
deformation steps a,b. \\
The proof is inductive with respect to the length $n$ of the graph. Assume we have shown already that 
I-homotopy and homotopy are equivalent for graphs of length smaller than $n$, we
show it for graphs $G$ of length $n$.  \\
Let $x,z$ be vertices in $G$ attached to a contractible graph $y = S(x) \cap S(y)$.
In order to add the edge $xz$, we first make a pyramid extension $v$ over $x \cup y$.
Now use steps a) and b) to expand the vertex $v$ to become a copy of $y$. 
This is possible by induction because $y$ has length smaller than $n$. 
Now remove vertices from the $y$ until it is a point. Also this is is possible because
the old $y$ has length smaller than $n$. Finally we remove the old $y$. 
Now $x$ and $z$ are connected.  
\end{proof} 

Other homotopy notions have been considered for graphs. A homotopy notion due to Atkin
is more algebraic and is studied in \cite{Babson}.
A weaker equivalence relation called $s$-homotopy is introduced by \cite{BouletFieuxJouve}.
They ask the question whether it is equivalent to I-homotopy. 


\vspace{12pt}
\bibliographystyle{plain}

\end{document}